\documentclass[a4paper,11pt]{amsart}

\usepackage[T1]{fontenc}
\usepackage{amsmath,amssymb,amsthm,mathtools}
\usepackage{mathrsfs}
\usepackage{xcolor}
\usepackage{cite}
\usepackage[colorlinks=true,linkcolor=blue,citecolor=blue,urlcolor=blue]{hyperref}
\usepackage{aliascnt}
\usepackage[nameinlink,capitalize]{cleveref}
\usepackage{microtype}
\usepackage[margin=3cm]{geometry}

\allowdisplaybreaks
\numberwithin{equation}{section}

\theoremstyle{plain}
\newtheorem{theorem}{Theorem}[section]
\newaliascnt{proposition}{theorem}
\newtheorem{proposition}[proposition]{Proposition}
\aliascntresetthe{proposition}
\newaliascnt{lemma}{theorem}
\newtheorem{lemma}[lemma]{Lemma}
\aliascntresetthe{lemma}
\newaliascnt{corollary}{theorem}
\newtheorem{corollary}[corollary]{Corollary}
\aliascntresetthe{corollary}
\newaliascnt{conjecture}{theorem}
\newtheorem{conjecture}[conjecture]{Conjecture}
\aliascntresetthe{conjecture}
\newaliascnt{remark}{theorem}
\newtheorem{remark}[remark]{Remark}
\aliascntresetthe{remark}

\crefname{theorem}{theorem}{theorems}
\Crefname{theorem}{Theorem}{Theorems}
\crefname{proposition}{proposition}{propositions}
\Crefname{proposition}{Proposition}{Propositions}
\crefname{lemma}{lemma}{lemmas}
\Crefname{lemma}{Lemma}{Lemmas}
\crefname{corollary}{corollary}{corollaries}
\Crefname{corollary}{Corollary}{Corollaries}
\crefname{conjecture}{conjecture}{conjectures}
\Crefname{conjecture}{Conjecture}{Conjectures}
\crefname{remark}{remark}{remarks}
\Crefname{remark}{Remark}{Remarks}

\newcommand{\Sph}{\mathbb S}
\newcommand{\Ball}{\mathbb B}
\newcommand{\R}{\mathbb R}
\newcommand{\Vol}{\operatorname{Vol}}
\newcommand{\inj}{\operatorname{inj}}
\newcommand{\tr}{\operatorname{tr}}
\newcommand{\dd}{\,d}

\newif\ifmarked
\markedfalse
\newcommand{\markbegin}{\ifmarked\color{blue}\fi}
\newcommand{\markend}{\ifmarked\color{black}\fi}

\title[Volume gap for minimal submanifolds in spheres, II]{\markbegin Volume gap for minimal submanifolds in spheres, II\markend}

\author[J. Q. Ge]{Jianquan Ge}
\address{School of Mathematical Sciences, Laboratory of Mathematics and Complex Systems, Beijing Normal University, Beijing 100875, P. R. China}
\email{jqge@bnu.edu.cn}

\author[F. G. Li]{Fagui Li${}^{*}$}
\address{Frontier Interdisciplinary Domain, Beijing Institute of Technology, Zhuhai, Guangdong 519088, P. R. China}
\email{lifagui@bitzh.edu.cn}

\subjclass[2020]{53C42, 53C24, 53C40}
\keywords{Minimal submanifold, higher codimension, linearly full immersion, volume gap, multiplicity, spherical monotonicity, minimal cone, normal bundle, nodal domain, hyperplane partition}
\date{}
\thanks{* Corresponding author.}
\thanks{J. Q. Ge is partially supported by NSFC (No. 12571049) and the Fundamental Research Funds for the Central Universities. }
\thanks{F. G. Li is partially supported by NSFC (No. 12271040 and 12501061), the Guangdong Provincial Association for Science and Technology Youth Talent Support Program (No. SKXRC2026413) and the Research Start-up Funding of Beijing Institute of Technology (No. 5640011253301).}
\begin{document}

\begin{abstract}
\markbegin
Let $f:M^n\looparrowright\Sph^{n+q}(1)$, $n\ge2$ and $q\ge1$, be a closed, connected, non-totally-geodesic minimal immersion with second fundamental form $h$, and put $S=|h|^2$ and $S_*=\max_M S$.  If $p\in f(M)$ has multiplicity $m$ and $f^{-1}(p)=\{x_1,\ldots,x_m\}$, then
\[
 \Vol(M)\ge
 \left[m+\varepsilon_n\sum_{j=1}^m
 \left(\frac{S(x_j)}{S_*}\right)^2\right]\Vol(\Sph^n),
\]
where $[110n(n+2)^2]^{-1}<\varepsilon_n<[104n(n+2)^2]^{-1}$.  If the immersion is linearly full, then
\[
 \frac{\Vol(M)}{\Vol(\Sph^n)}
 \ge
 \max\!\left\{1+\varepsilon_n,
 \frac{4(n+1)^n}{(n+3)^{n+2}}(n+q+1)\right\}.
\]
Moreover, for every hyperplane $H$ through the origin, each connected component of $M\setminus f^{-1}(H)$ has volume at least
$4(n+1)^n(n+3)^{-n-2}\Vol(\Sph^n)$; consequently the number of components is at most
$\frac{(n+3)^{n+2}}{4(n+1)^n}\frac{\Vol(M)}{\Vol(\Sph^n)}$.
\end{abstract}

\maketitle

\section{Introduction}

Let
$
 f:M^n\looparrowright\Sph^{n+q}(1)\subset\R^{n+q+1}
$
be a closed minimal immersion with second fundamental form $h$.  The immersion is called \emph{linearly full} if its image is not contained in a proper linear subsphere; this is the ``maximal dimension'' hypothesis used in the classical literature.

The first universal linearly full volume gap was proved by Cheng, Li, and Yau using heat-kernel comparison.  Define
\[
 C_n^{\rm CLY}
 =\frac12 n^{n/2}e\,\Gamma\!\left(\frac n2,1\right),
 \qquad
 \Gamma(s,1)=\int_1^\infty t^{s-1}e^{-t}\dd t,
\]
and
\[
 B_n^{\rm CLY}=2n+3+2\exp\bigl(2nC_n^{\rm CLY}\bigr).
\]

\begin{theorem}[Cheng--Li--Yau \cite{ChengLiYau1984}]\label{thm:CLY}
If $M^n\looparrowright\Sph^{n+q}(1)$ is compact, minimal, and linearly full, then
\[
 \Vol(M)>
 \left(1+\frac{2q-1}{B_n^{\rm CLY}}\right)\Vol(\Sph^n).
\]
\end{theorem}

\markbegin
For comparison with another consequence of the same heat-kernel method, we recall the first assertion of Corollary~6 of Cheng--Li--Yau \cite{ChengLiYau1984}.  We rewrite it in the present notation, with $q$ denoting their codimension $l$.  For a topological space $X$, we write $\pi_0(X)$ for the set of connected components of $X$, and $\# A$ for the cardinality of a set $A$.

\begin{theorem}[Cheng--Li--Yau  \cite{ChengLiYau1984}]\label{thm:CLY-partition}
Let $f:M^n\looparrowright\Sph^{n+q}(1)$ be a compact minimal immersion of maximal dimension, equivalently a linearly full immersion.  For a hyperplane $H\subset\R^{n+q+1}$ through the origin, let
\[
 \mathcal N_H=\#\pi_0\bigl(M\setminus f^{-1}(H)\bigr),
\]
so that $\mathcal N_H$ is the number of connected components into which $H$ divides $M$.
Then
\[
 \mathcal N_H<
 \left(\frac{\Vol(M)}{\Vol(\Sph^n)}-1\right)
 \frac{B_n^{\rm CLY}}2+q+\frac32.
\]
\end{theorem}

Their proof uses that the coordinate height functions have eigenvalue $n$, bounds the spectral position of this eigenspace by the heat-trace estimate in their Theorem~6, and then applies the Courant nodal-domain theorem.  Thus both the maximal-dimension hypothesis and the rapidly growing factor $B_n^{\rm CLY}$ enter through the spectral argument.
\markend

Ding, Ge, and Li \cite{DingGeLi2025} sharpened the coefficients in the Cheng--Li--Yau \cite{ChengLiYau1984} argument and subsequently combined the heat-trace comparison with the Cheng--Yang eigenvalue estimate \cite{ChY} to obtain a further improvement.  The present paper is a continuation of our previous work \cite{GeLi2022} and builds on the spherical height-function monotonicity and multiplicity estimates developed there.  Here the codimension factor is recovered from the full space of spherical height functions and a mean-value formula on the Euclidean minimal cone, without a heat kernel or heat trace.

\markbegin
A related benchmark is Yau's second-smallest-volume conjecture for closed minimal hypersurfaces \cite{Yau1992}.  For $1\le k\le n-1$, let
\[
 \mathcal C_{k,n-k}
 =\Sph^k\!\left(\sqrt{\frac{k}{n}}\right)
 \times
 \Sph^{n-k}\!\left(\sqrt{\frac{n-k}{n}}\right)
 \subset \Sph^{n+1}(1)
\]
be the minimal Clifford hypersurface.

\begin{conjecture}[Yau's second-smallest-volume conjecture \cite{Yau1992}]\label{conj:Yau-volume}
Every closed non-totally-geodesic minimal hypersurface $M^n\looparrowright\Sph^{n+1}(1)$ satisfies
\[
 \Vol(M)\ge
 \min_{1\le k\le n-1}\Vol(\mathcal C_{k,n-k}).
\]
Equivalently, after the totally geodesic equator, the least possible volume should be attained by a minimal Clifford hypersurface.
\end{conjecture}

For $n=2$, \Cref{conj:Yau-volume} follows from the work of Li--Yau, Calabi, and Marques--Neves \cite{LiYau1982,Calabi1967,MarquesNeves2014}; Brendle's proof of the Lawson conjecture gives the corresponding uniqueness theorem for embedded minimal tori in $\Sph^3$ \cite{Brendle2013}.  The conjectural lower bound is known for several classes of rotational hypersurfaces \cite{PerdomoWei2015,ChengWeiZeng2021,ChengLaiWei2024}; Ilmanen--White \cite{IlmanenWhite2015} obtained the corresponding sharp asymptotic density statement for certain area-minimizing cones; and Viana \cite{Viana2023} proved the lower bound in an antipodally invariant class.  The nonembedded case was established by Ge--Li \cite{GeLi2022} and Nguyen \cite{Nguyen2023}.  The multiplicity estimates behind those results originate in the spherical monotonicity method of Choe--Gulliver \cite{ChoeGulliver1992} and its weighted refinements \cite{Nguyen2023,GeLi2022}.
\markend

The curvature part of our argument belongs to the Simons-gap tradition.  Simons \cite{Simons1968} established the fundamental equation for the squared norm of the second fundamental form $h$. The equality cases of the first gap were analyzed by Lawson~\cite{Lawson1969} and Chern, do Carmo, and Kobayashi~\cite{ChernDoCarmoKobayashi1970}, while Yau~\cite{Yau1975} sharpened the codimension-two picture.  Li--Li \cite{LiLi1992} and Chen--Xu~\cite{ChenXu1993} obtained the universal higher-codimensional threshold $2n/3$ of $S=|h|^2$; and Lu \cite{Lu2011}  connected the normal scalar-curvature inequality with the DDVV framework.  A recent explicit higher-codimensional second-gap theorem under flatness of the normal bundle is due to Ge--Li--Zhang \cite{GeLiZhang2026}.  \markbegin
Very recently, Li--Zhao \cite{LiZhaoCurvatureGap2026} proved that every closed, connected, non-totally-geodesic minimal immersion $M^n\looparrowright\Sph^{n+q}(1)$ with $n\ge3$ and $q\ge2$ satisfies
$
 \max_M |h|^2\ge \frac{2n}{3}
 +\frac{1}{787500},
$
which gives an explicit separation from the Li--Li and Chen--Xu first-gap level $2n/3$.

  For $p\in f(M)$, let
\[
 \varphi_p=\langle f,p\rangle,
 \qquad
 \xi_p=p^\perp,
\]
where $p^\perp$ is the component of the fixed Euclidean vector $p$ in the normal bundle of $M$ inside the sphere.  The basic identity proved in \Cref{prop:exact-defect} is
\[
 \int_{\{\varphi_p\ge0\}}\varphi_p\dd V
 =m\Vol(\Ball^n)+\mathcal R_p,
 \qquad
 \mathcal R_p=
 \int_{\{\varphi_p>0\}}
 \frac{\varphi_p|\xi_p|^2}{(1-\varphi_p^2)^{n/2+1}}\dd V,
\]
where $m=\#f^{-1}(p)$.  Keeping both this local defect and the global normal-height energy gives
\[
 \Vol(M)\ge m\Vol(\Sph^n)
 +\frac{\Vol(\Sph^n)}{\Vol(\Ball^n)}\mathcal R_p
 +\frac1{n+2}\int_M|\xi_p|^2\dd V.
\]

For the curvature-weighted estimate, put
\[
 S=|h|^2,
 \qquad
 S_*=\max_M S.
\]
The proof below produces an explicit constant $\varepsilon_n>0$ depending only on the dimension.  Its exact formula is given in \eqref{eq:epsilon-definition}; for the statement of the main theorem it is enough to note that
\[
 \frac{1}{110n(n+2)^2}<\varepsilon_n<
 \frac{1}{104n(n+2)^2}.
\]
\begin{theorem}\label{thm:main}
Let $n\ge2$ and $q\ge1$, and let
$f:M^n\looparrowright\Sph^{n+q}(1)$ be a closed, connected, non-totally-geodesic minimal immersion.  
If $p\in f(M)$ and $f^{-1}(p)=\{x_1,\ldots,x_m\}$, then
\[
 \Vol(M)\ge
 \left[
 m+\varepsilon_n\sum_{j=1}^m
 \left(\frac{S(x_j)}{S_*}\right)^2
 \right]\Vol(\Sph^n).
\]
Here $\varepsilon_n$ is the codimension-independent coefficient defined explicitly in \eqref{eq:epsilon-definition}.
\end{theorem}

\markbegin
\begin{remark}\label{rem:q-one}
No separate argument is needed when $q=1$.  In that case there is only one shape operator, so the normal commutator terms vanish and the quartic term in Simons' identity is exactly $R_1=S^2$.  Hence the estimates $R_1\le \frac32S^2$ and $S_*\ge\frac{2n}{3}$ used in the codimension-free proof remain valid; in fact the hypersurface Simons identity gives the sharper first-gap bound $S_*\ge n$ for every non-totally-geodesic closed minimal hypersurface.  Thus \Cref{thm:main} holds for all $q\ge1$ with the same coefficient $\varepsilon_n$.

Moreover, a closed, connected, non-totally-geodesic minimal hypersurface in $\Sph^{n+1}$ is automatically linearly full.  Indeed, if its image were contained in a proper linear subsphere, the immersion dimension forces that subsphere to be an equatorial $\Sph^n$; the image is then a closed open subset of that equator and hence the immersion is totally geodesic.  Consequently, the linearly full results below also apply automatically in codimension one.
\end{remark}

If $x_*\in M$ satisfies $S(x_*)=S_*$ and $p_*=f(x_*)$, define
\[
 m_*=\#f^{-1}(p_*),
 \qquad
 \mathcal K_*=\sum_{x\in f^{-1}(p_*)}
 \left(\frac{S(x)}{S_*}\right)^2.
\]
Then $m_*\ge1$ and $\mathcal K_*\ge1$.

For the linearly full result, set
\[
 b_n=\frac{4(n+1)^n}{(n+3)^{n+2}}.
\]

\begin{theorem}\label{thm:full-main}
Under the hypotheses of \Cref{thm:main}, assume in addition that the immersion is linearly full.  Then
\begin{equation}\label{eq:full-max}
 {
 \frac{\Vol(M)}{\Vol(\Sph^n)}
 \ge
 \max\!\left\{
 m_*+\varepsilon_n\mathcal K_*,\,
 b_n(n+q+1)
 \right\}
 \ge 
  \max\!\left\{
  1+\varepsilon_n,\,
  b_n(n+q+1)
  \right\}.}
\end{equation}
\end{theorem}

The cone mean-value argument can also be localized to a single nodal domain of one height function.  For $a\in\Sph^{n+q}$, let
\[
 \mathcal H_a=\{y\in\R^{n+q+1}:\langle y,a\rangle=0\}
\]
be the hyperplane through the origin orthogonal to $a$, and define
\[
 \mathcal N_a=\#\pi_0\bigl(M\setminus f^{-1}(\mathcal H_a)\bigr).
\]

\begin{theorem}\label{thm:hyperplane-partition}
Let $f:M^n\looparrowright\Sph^{n+q}(1)$ be a closed minimal immersion, where $n\ge2$ and $q\ge1$.   For every $a\in\Sph^{n+q}$ and every connected component $\Omega$ of $M\setminus f^{-1}(\mathcal H_a)$,
\[
 \Vol(\Omega)\ge b_n\Vol(\Sph^n),
 \qquad
 b_n=\frac{4(n+1)^n}{(n+3)^{n+2}}.
\]
Consequently, $\mathcal N_a$ is finite and
\[
 \mathcal N_a
 \le
 \frac{(n+3)^{n+2}}{4(n+1)^n}
 \frac{\Vol(M)}{\Vol(\Sph^n)}
 <
 \frac{e^2}{4}(n+3)^2
 \frac{\Vol(M)}{\Vol(\Sph^n)}.
\]
If $f(M)\subset\mathcal H_a$, then $\mathcal N_a=0$ and the assertion is vacuous.
\end{theorem}

\begin{remark}\label{rem:CLY-partition-comparison}
The conclusions of \Cref{thm:CLY-partition,thm:hyperplane-partition} have different strengths.  Cheng--Li--Yau control the total number of nodal domains through the spectral position of the eigenvalue $n$ and require maximal dimension.  By contrast, \Cref{thm:hyperplane-partition} requires no linear fullness assumption and first gives the componentwise estimate
\[
 \Vol(\Omega)\ge b_n\Vol(\Sph^n).
\]
Only afterwards do we sum over the components.  Moreover,
\[
 b_n^{-1}
 =\frac{(n+3)^{n+2}}{4(n+1)^n}
 <\frac{e^2}{4}(n+3)^2,
\]
so the coefficient multiplying the total volume has polynomial growth of order $n^2$, whereas $B_n^{\rm CLY}/2$ contains the factor $\exp(2nC_n^{\rm CLY})$.  Since the Cheng--Li--Yau estimate is expressed in terms of $\Vol(M)/\Vol(\Sph^n)-1$ and also contains the additive term $q+3/2$, the two numerical upper bounds are not asserted to be pointwise ordered for every small volume ratio; the improvement here is the componentwise geometric estimate, the removal of the maximal-dimension assumption, and the polynomial dependence on the dimension.
\end{remark}

For a normalized additive form, define
\begin{equation}\label{eq:full-constants}
 \delta_n=
 \frac{\varepsilon_n b_n}
 {1+\varepsilon_n-(n+1)b_n},
 \qquad
 \rho_n=
 \frac{1-(n+1)b_n}
 {1+\varepsilon_n-(n+1)b_n}.
\end{equation}

\begin{corollary}\label{cor:full-additive}
Under the hypotheses of \Cref{thm:full-main},
\begin{equation}\label{eq:full-additive}
 {
 \frac{\Vol(M)}{\Vol(\Sph^n)}
 \ge
 1+\delta_nq
 +\rho_n\left[(m_*-1)+\varepsilon_n(\mathcal K_*-1)\right].}
\end{equation}
In particular,
$
 \Vol(M)\ge(1+\delta_nq)\Vol(\Sph^n)
$
and
\begin{equation}\label{eq:delta-polynomial-intro}
 \delta_n>
 \frac{1}{440n(n+2)^2(n+3)^2}
 >\frac{1}{11000n^5}.
\end{equation}
\end{corollary}

\begin{remark}\label{rem:CLY-comparison}
The main max estimate \eqref{eq:full-max} can be substantially stronger in high codimension.  To compare bounds having the same normalized baseline $1$, we compare the codimension term in \Cref{cor:full-additive} with the explicit Cheng--Li--Yau term.  The ratio is
\[
 \frac{\delta_nq}{(2q-1)/B_n^{\rm CLY}}
 =\delta_nB_n^{\rm CLY}\frac{q}{2q-1}.
\]
For $q=2$ the first values are
\[
\begin{array}{c|c|c|c}
 n&2\delta_n&3/B_n^{\rm CLY}&(2\delta_n)/(3/B_n^{\rm CLY})\\ \hline
 2&4.1505\times10^{-5}&2.5818\times10^{-2}&1.61\times10^{-3}\\
 3&9.2827\times10^{-6}&6.9298\times10^{-10}&1.34\times10^{4}\\
 4&3.0347\times10^{-6}&3.8583\times10^{-56}&7.87\times10^{49}.
\end{array}
\]
Thus the Cheng--Li--Yau bound is stronger in dimension two, whereas our explicit coefficient is larger by four orders of magnitude in dimension three and by much more in higher dimensions.  Ding--Ge--Li \cite{DingGeLi2025} improved the heat-kernel denominator and then enlarged the gap further by an eigenvalue argument.  Their displayed constants still contain a rapidly growing exponential factor in the dimension, while \eqref{eq:delta-polynomial-intro} has polynomial size.  This comparison concerns only the stated explicit constants; no sharpness claim is made for either universal gap.
\end{remark}

The proof has five ingredients.  First, the level-set derivative of the height function gives the exact normal defect.  Second, the full differentiated Simons identity and a self-contained higher-codimensional Peng--Terng estimate yield
\[
 \sup_M|\nabla h|\le D_nS_*.
\]
Third, a vector-valued second-jet estimate converts this derivative bound into a lower bound for $\mathcal R_p$ on each sheet above $p$.  Fourth, in the linearly full case, the one-homogeneous extensions of all height functions to the Euclidean minimal cone are used simultaneously; their mean-value inequality produces the factor $n+q+1$ without a heat kernel.  Finally, localizing the same cone argument to the weakly subharmonic zero extension of the squared height function on one nodal cone proves \Cref{thm:hyperplane-partition}.

The paper is organized accordingly.  Basic notation, geometric identities, the exact height defect, and the two-remainder inequality are collected in \Cref{sec:prelim}.  The curvature scale and derivative estimate are established in \Cref{sec:scale}.  The vector-valued sheet estimate appears in \Cref{sec:vector-defect}.  The linearly full argument is given in \Cref{sec:full}.  The hyperplane partition theorem is proved in \Cref{sec:hyperplane-partition}, and the numerical and asymptotic estimates are collected in \Cref{sec:constants}.

\markbegin
\section{Preliminaries}\label{sec:prelim}

\subsection{Geometric notation and basic identities}
Throughout the paper all manifolds and immersions are smooth.  
Let
$
 f:M^n\looparrowright \Sph^{n+q}(1)\subset\R^{n+q+1}
$
be a closed immersed submanifold with induced metric $g$.  We write
$\nabla$ for the Levi--Civita connection of $M$, $\nabla^\perp$ for the normal connection of $M$ in the sphere, and $D$ for the Euclidean connection.  The second fundamental form $h$ and the shape operator $A_\xi$ in a normal direction $\xi$ are defined by
\[
 D_XY=\nabla_XY+h(X,Y)-\langle X,Y\rangle f,
 \qquad
 \langle A_\xi X,Y\rangle=\langle h(X,Y),\xi\rangle.
\]
The mean-curvature vector is
\[
 H=\frac1n\sum_{i=1}^nh(e_i,e_i),
\]
where $\{e_i\}_{i=1}^n$ is any local orthonormal tangent frame.  The immersion is minimal precisely when $H=0$.

Choose local orthonormal tangent and normal frames $\{e_i\}_{i=1}^n$ and $\{\nu_\alpha\}_{\alpha=1}^q$.  We use
\[
 h_{ij}^\alpha=\langle h(e_i,e_j),\nu_\alpha\rangle,
 \qquad
 A_\alpha=A_{\nu_\alpha}=(h_{ij}^\alpha),
 \qquad
 S=|h|^2=\sum_{i,j,\alpha}(h_{ij}^\alpha)^2
 =\sum_\alpha|A_\alpha|^2.
\]
All matrix inner products and norms are Hilbert--Schmidt.  Covariant derivatives are denoted by
\[
 h_{ijk}^\alpha=(\nabla h)_{ijk}^\alpha,
 \qquad
 |\nabla h|^2=\sum_{i,j,k,\alpha}(h_{ijk}^\alpha)^2.
\]
For a minimal immersion, $\operatorname{tr}A_\alpha=0$ for every normal index $\alpha$.

We use the curvature convention
\[
 R(X,Y)Z=\nabla_X\nabla_YZ-\nabla_Y\nabla_XZ-\nabla_{[X,Y]}Z
\]
and the analogous convention for the normal curvature $R^\perp$.  With these conventions the Gauss, Codazzi, and Ricci equations are
\begin{align*}
 R_{ijkl}
 &=\delta_{ik}\delta_{jl}-\delta_{il}\delta_{jk}
 +\sum_\alpha\bigl(h_{ik}^\alpha h_{jl}^\alpha-h_{il}^\alpha h_{jk}^\alpha\bigr),\\
 h_{ijk}^\alpha&=h_{ikj}^\alpha,\\
 R^\perp_{\alpha\beta kl}
 &=\sum_i\bigl(h_{ik}^\alpha h_{il}^\beta-h_{il}^\alpha h_{ik}^\beta\bigr)
 =\langle[A_\alpha,A_\beta]e_k,e_l\rangle.
\end{align*}
These formulas fix the index and sign conventions used later in the differentiated Simons computation.

For volumes we write, for every integer $k\ge1$,
\[
 \omega_k=\Vol(\Sph^k)
 =\frac{2\pi^{(k+1)/2}}{\Gamma((k+1)/2)},
 \qquad
 \beta_k=\Vol(\Ball^k)
 =\frac{\pi^{k/2}}{\Gamma(k/2+1)}
 =\frac{\omega_{k-1}}k.
\]

\subsection{Spherical height functions and the exact normal defect}

For $a\in\Sph^{n+q}$, define
\[
 \varphi_a=\langle f,a\rangle.
\]
Takahashi's identities \cite{Takahashi1966} are
\begin{equation}\label{eq:height-identities}
 \nabla\varphi_a=a^T,
 \qquad
 \Delta\varphi_a=-n\varphi_a,
 \qquad
 \int_M\varphi_a\dd V=0.
\end{equation}
The fixed Euclidean vector $a$ has the pointwise orthogonal decomposition
\[
 a=\varphi_af+\nabla\varphi_a+\xi_a,
 \qquad
 \xi_a=a^\perp\in\Gamma(NM).
\]
Consequently,
\begin{equation}\label{eq:normal-defect-pointwise}
 1-\varphi_a^2-|\nabla\varphi_a|^2=|\xi_a|^2\ge0.
\end{equation}

For $0<t<1$, set
\[
 I_a(t)=\int_{\{\varphi_a\ge t\}}\varphi_a\dd V,
 \qquad
 \Theta_a(t)=\frac{I_a(t)}{(1-t^2)^{n/2}}.
\]

\begin{proposition}\label{prop:Theta}
At every regular value $t\in(0,1)$,
\begin{equation}\label{eq:Theta-derivative}
 \Theta_a'(t)
 =-\frac{t}{(1-t^2)^{n/2+1}}
 \int_{\{\varphi_a=t\}}\frac{|\xi_a|^2}{|\nabla \varphi_a|}\dd\sigma.
\end{equation}
In particular, $\Theta_a$ is nonincreasing.
\end{proposition}

\begin{proof}
Let $\Sigma_t=\{\varphi_a=t\}$ and orient it by the outward normal of $\{\varphi_a\ge t\}$, namely $-\nabla \varphi_a/|\nabla \varphi_a|$.  The divergence theorem and \eqref{eq:height-identities} give
\[
 nI_a(t)=\int_{\Sigma_t}|\nabla \varphi_a|\dd\sigma.
\]
The coarea formula gives, at every regular value,
\[
 I_a'(t)=-t\int_{\Sigma_t}\frac1{|\nabla \varphi_a|}\dd\sigma.
\]
Differentiating $\Theta_a$ and using these two identities yields
\begin{align*}
 \Theta_a'(t)
 &=\frac{t}{(1-t^2)^{n/2+1}}
 \left[
 \int_{\Sigma_t}|\nabla \varphi_a|\dd\sigma
 -(1-t^2)\int_{\Sigma_t}\frac1{|\nabla \varphi_a|}\dd\sigma
 \right]\\
 &=-\frac{t}{(1-t^2)^{n/2+1}}
 \int_{\Sigma_t}
 \frac{1-t^2-|\nabla \varphi_a|^2}{|\nabla \varphi_a|}\dd\sigma.
\end{align*}
Now use \eqref{eq:normal-defect-pointwise}.  More precisely, by the coarea formula, $I_a$ and hence $\Theta_a$ are locally absolutely continuous on every compact subinterval of $(0,1)$.  The derivative formula therefore holds for almost every $t$, and its right-hand side is nonpositive.  It follows that $\Theta_a$ is nonincreasing on $(0,1)$.
\end{proof}

We shall use the following height comparison from \cite[Theorem~3.5]{GeLi2022}.
\begin{lemma}[Ge--Li \cite{GeLi2022}]\label{lem:height-comparison}
For every $a\in\Sph^{n+q}$,
\[
 \int_{\{\varphi_a\ge0\}}(1+\varphi_a^2)\dd V
 \ge
 \frac{(n+2)\omega_n}{2(n+1)\beta_n}
 \int_{\{\varphi_a\ge0\}}\varphi_a\dd V.
\]
\end{lemma}

\begin{lemma}\label{lem:endpoint-density}
Let $p\in f(M)$ and suppose
$f^{-1}(p)=\{x_1,\ldots,x_m\}$.  For $\varphi_p$,
\begin{equation}\label{eq:endpoint-density}
 \lim_{t\to1^-}\frac{\int_{\{\varphi_p\ge t\}}\varphi_p\dd V}
 {(1-t^2)^{n/2}}=m\beta_n.
\end{equation}
\end{lemma}

\begin{proof}
Because $f$ is an immersion, every point of $f^{-1}(p)$ is isolated: in a sufficiently small coordinate neighborhood the map is an embedding.  Thus $f^{-1}(p)$ is a discrete closed subset of the compact manifold $M$, and is therefore finite.  At every $x_j$ one has
\[
 \varphi_p(x_j)=1,
 \qquad
 \nabla \varphi_p(x_j)=0,
 \qquad
 \xi_p(x_j)=0.
\]
The Hessian formula proved in \Cref{lem:normal-identities} below gives $\nabla^2\varphi_p(x_j)=-g$.  Hence each $x_j$ is a nondegenerate maximum.  In geodesic normal coordinates $y$ centered at $x_j$,
\[
 \varphi_p(\exp_{x_j}y)=1-\frac{|y|^2}{2}+O(|y|^3),
 \qquad
 \dd V=(1+O(|y|^2))\dd y.
\]
For $t$ sufficiently close to $1$, the set $\{\varphi_p\ge t\}$ is the disjoint union of $m$ neighborhoods of the points $x_j$.  Put $\rho=(1-t^2)^{1/2}$ and rescale $y=\rho z$.  Since
\[
 t=1-\frac{\rho^2}{2}+O(\rho^4),
\]
the rescaled component of $\{\varphi_p\ge t\}$ converges to the Euclidean unit ball.  The displayed Taylor expansions and dominated convergence therefore give
\[
 \rho^{-n}\int_{\{\varphi_p\ge t\}\cap U_j}\varphi_p\dd V\longrightarrow\beta_n
\]
for every $j$.  Summing the $m$ components proves \eqref{eq:endpoint-density}.
\end{proof}

\begin{proposition}\label{prop:exact-defect}
Let $p\in f(M)$ and $f^{-1}(p)=\{x_1,\ldots,x_m\}$.  Then
\begin{equation}\label{eq:exact-defect}
 \int_{\{\varphi_p\ge0\}}\varphi_p\dd V
 =m\beta_n+\mathcal R_p,
\end{equation}
where
\begin{equation}\label{eq:defect-definition}
 \mathcal R_p=
 \int_{\{\varphi_p>0\}}
 \frac{\varphi_p|p^\perp|^2}
 {(1-\varphi_p^2)^{n/2+1}}\dd V.
\end{equation}
The integral is finite and nonnegative.
\end{proposition}

\begin{proof}
  Integrating \eqref{eq:Theta-derivative} between regular values $0<a<b<1$ and applying the coarea formula gives
\[
 \Theta_p(a)-\Theta_p(b)
 =\int_{\{a<\varphi_p<b\}}
 \frac{\varphi_p|p^\perp|^2}{(1-\varphi_p^2)^{n/2+1}}\dd V.
\]
Choose regular sequences $a_k\downarrow0$ and $b_k\uparrow1$.  Since the integrand $\varphi_p$ vanishes on $\{\varphi_p=0\}$, monotone convergence gives
$\Theta_p(a_k)\to\int_{\{\varphi_p\ge0\}}\varphi_p\,\dd V$, while \Cref{lem:endpoint-density} gives $\Theta_p(b_k)\to m\beta_n$.  The integrands on the right are nonnegative and the domains increase to $\{0<\varphi_p<1\}$.  The omitted level $\{\varphi_p=1\}=f^{-1}(p)$ is finite and hence has zero measure.  Monotone convergence therefore gives \eqref{eq:exact-defect} and simultaneously proves the finiteness of \eqref{eq:defect-definition}.
\end{proof}

\begin{theorem}\label{thm:two-remainders}
For every multiplicity-$m$ point $p\in f(M)$,
\[
 \Vol(M)\ge m\omega_n
 +\frac{\omega_n}{\beta_n}\mathcal R_p
 +\frac1{n+2}\int_M|p^\perp|^2\dd V.
\]
\end{theorem}

\begin{proof}
Apply \Cref{lem:height-comparison} to $\varphi_p$ and to $-\varphi_p$.  Since $\varphi_p=1$ at every point of $f^{-1}(p)$, the function is not identically zero.  The standard nodal-set theorem for a nontrivial solution of the elliptic equation $\Delta\varphi_p+n\varphi_p=0$ implies that $\{\varphi_p=0\}$ has zero $n$-dimensional measure.  Since $\int_M\varphi_p\dd V=0$, the positive and negative parts of $\varphi_p$ have the same integral, and addition gives
\[
 \frac{n+1}{n+2}\int_M(1+\varphi_p^2)\dd V
 \ge\frac{\omega_n}{\beta_n}\int_{\{\varphi_p\ge0\}}\varphi_p\dd V.
\]
On the other hand, integration by parts in \eqref{eq:height-identities} gives
\[
 \int_M|\nabla \varphi_p|^2\dd V=n\int_M\varphi_p^2\dd V.
\]
Integrating \eqref{eq:normal-defect-pointwise} therefore yields
\[
 \Vol(M)=(n+1)\int_M\varphi_p^2\dd V+\int_M|p^\perp|^2\dd V.
\]
Consequently,
\begin{align*}
 \Vol(M)
 &=\frac{n+1}{n+2}
 \left(\Vol(M)+\int_M\varphi_p^2\dd V\right)
 +\frac1{n+2}\int_M|p^\perp|^2\dd V\\
 &\ge\frac{\omega_n}{\beta_n}
 \int_{\{\varphi_p\ge0\}}\varphi_p\dd V
 +\frac1{n+2}\int_M|p^\perp|^2\dd V.
\end{align*}
Now use \Cref{prop:exact-defect}.
\end{proof}

\markend

\section{An explicit curvature scale}\label{sec:scale}

We use the notation and curvature conventions fixed in \Cref{sec:prelim}.  In particular,
$A_\alpha=(h_{ij}^\alpha)$, $S=|h|^2=\sum_\alpha|A_\alpha|^2$, and
$\operatorname{tr}A_\alpha=0$ for a minimal immersion.  Define
\[
 R_1=
 \sum_{\alpha,\beta}\langle A_\alpha,A_\beta\rangle^2
 +\sum_{\alpha,\beta}|[A_\alpha,A_\beta]|^2.
\]
The Simons identity and the Li--Li matrix inequality are
\begin{equation}\label{eq:Simons-LiLi}
 \frac12\Delta S=|\nabla h|^2+nS-R_1,
 \qquad
 R_1\le\frac32S^2;
\end{equation}
see \cite{Simons1968,LiLi1992,Lu2011}.  At a maximum point of $S$, these formulas imply that every non-totally-geodesic immersion satisfies
\begin{equation}\label{eq:first-curvature-gap}
 S_*\ge\frac{2n}{3}.
\end{equation}

\subsection{Trace-free estimates and separation of sheets}

\begin{lemma}\label{lem:trace-free}
For every unit tangent vector $X$,
\[
 \sum_\alpha|A_\alpha X|^2\le\frac{n-1}{n}S.
\]
For every orthonormal pair $X,Y$,
\[
 \langle h(X,X),h(Y,Y)\rangle
 \le\frac{n-2}{2n}S.
\]
When $n=2$, the right side of the second inequality is zero.
\end{lemma}

\begin{proof}
For a trace-free symmetric matrix $A$, choose an orthonormal basis with $X=e_1$ and write
\[
 A=\begin{pmatrix}a&v^T\\v&B\end{pmatrix},
 \qquad
 \tr B=-a.
\]
Then
\[
 |AX|^2=a^2+|v|^2,
 \qquad
 |A|^2=a^2+2|v|^2+|B|^2.
\]
Since $|B|^2\ge a^2/(n-1)$ and $2\ge n/(n-1)$,
\[
 |A|^2\ge\frac{n}{n-1}(a^2+|v|^2).
\]
Apply this to each $A_\alpha$ and sum over $\alpha$.

For the second inequality, fix $\alpha$ and put
\[
 a_\alpha=\langle A_\alpha X,X\rangle,
 \qquad
 b_\alpha=\langle A_\alpha Y,Y\rangle.
\]
If $a_\alpha b_\alpha\le0$, then
$a_\alpha b_\alpha\le[(n-2)/(2n)]|A_\alpha|^2$ is immediate.  If $a_\alpha b_\alpha>0$ and $n\ge3$, trace freeness and Cauchy--Schwarz on the orthogonal complement of $\operatorname{span}\{X,Y\}$ give
\[
 |A_\alpha|^2
 \ge a_\alpha^2+b_\alpha^2
 +\frac{(a_\alpha+b_\alpha)^2}{n-2}
 \ge\frac{2n}{n-2}a_\alpha b_\alpha.
\]
Summing over $\alpha$ proves the result.  If $n=2$, trace freeness gives $b_\alpha=-a_\alpha$.
\end{proof}

Set
\[
 \kappa_n=\frac{n+1}{2n},
 \qquad
 \lambda_n=\frac{2n+1}{2n}.
\]

\begin{lemma}\label{lem:scale-separation}
The induced metric satisfies
\begin{equation}\label{eq:sectional-upper}
 \operatorname{sec}_M\le\kappa_nS_*.
\end{equation}
Moreover,
\begin{equation}\label{eq:injectivity-lower}
 \inj(M)\ge\frac{\pi}{2\sqrt{\lambda_nS_*}}.
\end{equation}
If $x\ne y$ and $f(x)=f(y)$, then
\begin{equation}\label{eq:sheet-separation}
 d_M(x,y)\ge\frac{\pi}{\sqrt{\lambda_nS_*}}.
\end{equation}
\end{lemma}

\begin{proof}
For orthonormal $X,Y$, the Gauss equation and \Cref{lem:trace-free} give
\[
 K_M(X,Y)
 =1+\langle h(X,X),h(Y,Y)\rangle-|h(X,Y)|^2
 \le1+\frac{n-2}{2n}S_*.
\]
By \eqref{eq:first-curvature-gap}, $1\le3S_*/(2n)$, and hence \eqref{eq:sectional-upper} follows.  Rauch comparison \cite{Rauch1951} gives conjugate radius at least $\pi/\sqrt{\kappa_nS_*}$.

Let $\gamma$ be a unit-speed geodesic in $M$.  As a curve in Euclidean space,
\[
 \left|\frac{D}{\dd s}(f\circ\gamma)'\right|^2
 =|-f+h(\gamma',\gamma')|^2
 =1+|h(\gamma',\gamma')|^2.
\]
By \Cref{lem:trace-free} and \eqref{eq:first-curvature-gap},
\[
 1+|h(\gamma',\gamma')|^2
 \le1+\frac{n-1}{n}S_*
 \le\lambda_nS_*.
\]
Suppose that $\gamma$ is a nontrivial geodesic loop of length $L$.  Its Euclidean image is a closed piecewise smooth curve with one corner.  The corner angle is at most $\pi$, while Fenchel's theorem \cite{Fenchel1951} says that the total curvature, including the corner, is at least $2\pi$.  Hence
\[
 2\pi\le\sqrt{\lambda_nS_*}\,L+\pi,
\]
so $L\ge\pi/\sqrt{\lambda_nS_*}$.  Klingenberg's injectivity-radius lemma \cite{Klingenberg1959} states that the injectivity radius is at least the smaller of the conjugate radius and half the length of the shortest nontrivial geodesic loop.  Since $\kappa_n<\lambda_n$, this proves \eqref{eq:injectivity-lower}.

Finally, let $\gamma$ minimize the intrinsic distance between two distinct points $x,y$ with $f(x)=f(y)$.  Its Euclidean image is again a closed curve with a single corner of angle at most $\pi$.  The same Fenchel estimate, now without the factor $1/2$ from Klingenberg's lemma, proves \eqref{eq:sheet-separation}.
\end{proof}

\subsection{The differentiated Simons identity}

\markbegin
For the estimates in this subsection, set
\[
 \gamma_n=
 \begin{cases}
 \dfrac32,& n=2,\\[2mm]
 \dfrac{n+3+\sqrt{17n^2-26n+9}}{2n},& n\ge3,
 \end{cases}
 \qquad
 L_n=13+\gamma_n.
\]
\markend

At a point where the tangent and normal frames are synchronous, write
$\nabla_kA_\alpha=\nabla_{e_k}A_\alpha$.  For normal indices $\alpha,\beta$, define
\[
 \mathcal A_{\alpha\beta}
 =\sum_{i,j,k,l,p}h_{ijk}^{\beta}h_{ijp}^{\beta}
 h_{pl}^{\alpha}h_{lk}^{\alpha},
\]
\[
 \mathcal B_{\alpha\beta}
 =\sum_{i,j,k,l,p}h_{ijk}^{\beta}h_{lpk}^{\beta}
 h_{pj}^{\alpha}h_{il}^{\alpha}.
\]

The following identity is Proposition~2.1 of Ge--Li--Zhang \cite{GeLiZhang2026}.
\begin{lemma}[Ge--Li--Zhang \cite{GeLiZhang2026}]\label{lem:differentiated-Simons}
Every minimal immersion satisfies
\begin{align}\label{eq:differentiated-Simons}
 \frac12\Delta|\nabla h|^2
 ={}&|\nabla^2h|^2+(2n+3)|\nabla h|^2
 -3\sum_{\alpha,\beta}
 (\mathcal A_{\alpha\beta}-2\mathcal B_{\alpha\beta})\notag\\
 &-6\sum_{k,\alpha,\beta}
 \langle A_\alpha,\nabla_kA_\beta\rangle^2\notag\\
 &-\sum_{\alpha,\beta}\langle A_\alpha,A_\beta\rangle
 \sum_k\langle\nabla_kA_\alpha,\nabla_kA_\beta\rangle\notag\\
 &+3\sum_{\alpha,\beta}
 \left\langle[A_\alpha,A_\beta],
 \sum_k[\nabla_kA_\beta,\nabla_kA_\alpha]\right\rangle.
\end{align}
\end{lemma}

\begin{proof}
The differentiated Simons formula needed below is the identity recorded in \cite[Proposition~2.1]{GeLiZhang2026}.    For completeness, we fix the conventions and reproduce the raw identity, and then give the regrouping that leads to \eqref{eq:differentiated-Simons}.

Fix a point and choose tangent and normal orthonormal frames whose connection forms vanish at that point.  We use the following conventions for the Ricci identities:
\begin{align*}
 h_{ijkl}^{\alpha}-h_{ijlk}^{\alpha}
 &=h_{ip}^{\alpha}R_{pjkl}+h_{pj}^{\alpha}R_{pikl}
   -h_{ij}^{\beta}R^{\perp}_{\alpha\beta kl},\\
 h_{ijklm}^{\alpha}-h_{ijkml}^{\alpha}
 &=h_{pjk}^{\alpha}R_{pilm}+h_{ipk}^{\alpha}R_{pjlm}
   +h_{ijp}^{\alpha}R_{pklm}
   -h_{ijk}^{\beta}R^{\perp}_{\alpha\beta lm}.
\end{align*}
The Gauss and normal Ricci equations are
\begin{align*}
 R_{ijkl}
 &=\delta_{ik}\delta_{jl}-\delta_{il}\delta_{jk}
   +\sum_{\alpha}
   \bigl(h_{ik}^{\alpha}h_{jl}^{\alpha}
         -h_{il}^{\alpha}h_{jk}^{\alpha}\bigr),\\
 R^{\perp}_{\alpha\beta kl}
 &=\sum_i\bigl(h_{ik}^{\alpha}h_{il}^{\beta}
                -h_{il}^{\alpha}h_{ik}^{\beta}\bigr).
\end{align*}
Codazzi symmetry and minimality give
$h_{ijk}^{\alpha}=h_{ikj}^{\alpha}$ and
$\sum_i h_{iik}^{\alpha}=0$.  The ordinary Simons equation can be written in matrix form as
\[
 \Delta A_{\alpha}
 =nA_{\alpha}
 -\sum_{\beta}\langle A_{\alpha},A_{\beta}\rangle A_{\beta}
 -\sum_{\beta}[A_{\beta},[A_{\beta},A_{\alpha}]].
\]
Moreover,
\[
 \frac12\Delta|\nabla h|^2
 =|\nabla^2h|^2
  +\sum_{i,j,k,\alpha}h_{ijk}^{\alpha}\Delta h_{ijk}^{\alpha}.
\]
Differentiating the matrix Simons equation covariantly in the $e_k$ direction, commuting the derivative through the rough Laplacian with the two displayed Ricci identities, and then substituting the Gauss and normal Ricci equations gives the unrestricted formula of \cite[Proposition~2.1]{GeLiZhang2026}.  Codazzi symmetry and differentiated minimality cancel the trace terms, while the linear ambient-curvature terms combine to $(2n+3)|\nabla h|^2$.  In our index conventions the resulting raw identity is
\begin{align*}
 \frac12\Delta|\nabla h|^2
 ={}&|\nabla^2h|^2+(2n+3)|\nabla h|^2\\
 &+\sum_{i,j,k,p,l,\alpha,\beta}
 \bigl(
 6h_{ijk}^{\alpha}h_{lpk}^{\alpha}h_{pj}^{\beta}h_{il}^{\beta}
 -3h_{ijk}^{\alpha}h_{ijp}^{\alpha}h_{pl}^{\beta}h_{lk}^{\beta}\\
 &\qquad
 -6h_{ijk}^{\alpha}h_{pjl}^{\alpha}h_{pl}^{\beta}h_{ik}^{\beta}
 +6h_{ijk}^{\alpha}h_{lp}^{\alpha}h_{pjk}^{\beta}h_{il}^{\beta}\\
 &\qquad
 -6h_{ijk}^{\alpha}h_{pi}^{\alpha}h_{pl}^{\beta}h_{jlk}^{\beta}
 -h_{ijk}^{\alpha}h_{pl}^{\alpha}h_{pl}^{\beta}h_{ijk}^{\beta}
 \bigr).
\end{align*}
We now regroup every quartic term, without imposing any normal-flatness condition.  The first and second raw terms, after interchanging the dummy normal indices in the first one and matching the definitions of $\mathcal A_{\alpha\beta}$ and $\mathcal B_{\alpha\beta}$, contribute
\[
 -3\sum_{\alpha,\beta}
 (\mathcal A_{\alpha\beta}-2\mathcal B_{\alpha\beta}).
\]
The full symmetry of $h_{ijk}^{\alpha}$ and differentiated minimality give
\[
 \sum_{i,j,k,p,l,\alpha,\beta}
 h_{ijk}^{\alpha}h_{pjl}^{\alpha}h_{pl}^{\beta}h_{ik}^{\beta}
 =\sum_{k,\alpha,\beta}
 \langle A_\alpha,\nabla_kA_\beta\rangle^2.
\]
The last term of the raw formula is exactly
\[
 -\sum_{\alpha,\beta}\langle A_\alpha,A_\beta\rangle
 \sum_k\langle\nabla_kA_\alpha,\nabla_kA_\beta\rangle.
\]
For the remaining two terms, matrix multiplication and cyclicity of the trace give
\[
 6\sum_{\alpha,\beta}
 \left\langle[A_\alpha,A_\beta],
 \sum_k\nabla_kA_\beta\,\nabla_kA_\alpha\right\rangle.
\]
If $K$ is skew-symmetric and $P,Q$ are symmetric, then
\[
 2\langle K,PQ\rangle=\langle K,[P,Q]\rangle.
\]
Applying this identity with
$K=[A_\alpha,A_\beta]$, $P=\nabla_kA_\beta$, and
$Q=\nabla_kA_\alpha$ gives the last line of
\eqref{eq:differentiated-Simons}.
\end{proof}

For $n\ge3$, the next estimate is Lemma~3.4 of Ge--Li--Zhang \cite{GeLiZhang2026}; the case $n=2$ is verified directly below.
\begin{lemma}[Ge--Li--Zhang \cite{GeLiZhang2026}]\label{lem:PT}
For every fixed pair of normal indices $\alpha,\beta$,
\begin{equation}\label{eq:PT}
 3(\mathcal A_{\alpha\beta}-2\mathcal B_{\alpha\beta})
 \le\gamma_n|A_\alpha|^2|\nabla h^\beta|^2,
\end{equation}
where
$|\nabla h^\beta|^2=\sum_{i,j,k}(h_{ijk}^\beta)^2$ and $\gamma_n$ is the constant defined above.
\end{lemma}

\begin{proof}
For $n\ge3$ this is the algebraic estimate in \cite[Lemma~3.4]{GeLiZhang2026}; we include the short calculation because it also makes the exceptional case $n=2$ transparent.  Diagonalize
$A_\alpha=\operatorname{diag}(\lambda_1,\ldots,\lambda_n)$ and put
$T_{ijk}=h_{ijk}^\beta$.  By Codazzi symmetry and differentiated minimality, $T$ is fully symmetric and
$\sum_iT_{iik}=0$.  Direct contraction gives
\begin{align*}
 3(\mathcal A_{\alpha\beta}-2\mathcal B_{\alpha\beta})
 =\sum_{i,j,k}T_{ijk}^2\bigl[&\lambda_i^2+\lambda_j^2+\lambda_k^2
 -2(\lambda_i\lambda_j+\lambda_j\lambda_k+\lambda_k\lambda_i)\bigr].
\end{align*}
If $i,j,k$ are distinct, the coefficient in brackets is at most
$2(\lambda_i^2+\lambda_j^2+\lambda_k^2)\le2|A_\alpha|^2$.  If $i=j=k$, it equals $-3\lambda_i^2$ and is harmless.  If exactly two indices agree, the coefficient has the form
\[
 \lambda_j^2-4\lambda_i\lambda_j.
\]
For $n\ge3$, trace freeness yields
\[
 |A_\alpha|^2\ge
 \lambda_i^2+\lambda_j^2+
 \frac{(\lambda_i+\lambda_j)^2}{n-2}.
\]
The largest generalized eigenvalue in
\[
 \lambda_j^2-4\lambda_i\lambda_j
 \le\gamma\left(
 \lambda_i^2+\lambda_j^2+
 \frac{(\lambda_i+\lambda_j)^2}{n-2}
 \right)
\]
is the positive root of
\[
 n\gamma^2-(n+3)\gamma-4n+8=0,
\]
namely
\[
 \gamma_n=
 \frac{n+3+\sqrt{17n^2-26n+9}}{2n}.
\]
This number is larger than $2$, so it also controls the all-distinct case.  Summing the coefficient bounds against $T_{ijk}^2$ proves \eqref{eq:PT} for $n\ge3$.

It remains to verify $n=2$, which is not covered by the cited lemma.  Write
$A_\alpha=\operatorname{diag}(\lambda,-\lambda)$.  Full symmetry and
$\sum_iT_{iik}=0$ give
\[
 T_{111}=a,\qquad T_{112}=b,\qquad
 T_{122}=-a,\qquad T_{222}=-b.
\]
Substitution into the preceding contraction formula yields
\[
 3(\mathcal A_{\alpha\beta}-2\mathcal B_{\alpha\beta})
 =12\lambda^2(a^2+b^2),
\]
whereas
\[
 |A_\alpha|^2|\nabla h^\beta|^2
 =8\lambda^2(a^2+b^2).
\]
Thus \eqref{eq:PT} holds with the sharp coefficient
$\gamma_2=3/2$.
\end{proof}

\begin{proposition}\label{prop:bochner}
With $Q=|\nabla h|^2$,
\[
 \frac12\Delta Q
 \ge|\nabla^2h|^2+(2n+3-L_nS)Q.
\]
\end{proposition}

\begin{proof}
Summing \eqref{eq:PT} over $\alpha,\beta$ gives
\[
 3\sum_{\alpha,\beta}
 (\mathcal A_{\alpha\beta}-2\mathcal B_{\alpha\beta})
 \le\gamma_nSQ.
\]
For the second term in \eqref{eq:differentiated-Simons}, Cauchy--Schwarz gives
\[
 \sum_{k,\alpha,\beta}
 \langle A_\alpha,\nabla_kA_\beta\rangle^2
 \le SQ.
\]
The two matrices
\[
 (\langle A_\alpha,A_\beta\rangle)_{\alpha\beta},
 \qquad
 \left(\sum_k\langle\nabla_kA_\alpha,
 \nabla_kA_\beta\rangle\right)_{\alpha\beta}
\]
are positive semidefinite.  Their contraction is therefore nonnegative and at most the product of their traces, namely $SQ$.

Finally, the commutator inequality
$|[P,Q]|\le\sqrt2|P||Q|$ gives, for each $k$,
\begin{align*}
 &3\left|\sum_{\alpha,\beta}
 \langle[A_\alpha,A_\beta],
 [\nabla_kA_\beta,\nabla_kA_\alpha]\rangle\right|\\
 &\qquad\le6
 \left(\sum_\alpha|A_\alpha||\nabla_kA_\alpha|\right)^2
 \le6S\sum_\alpha|\nabla_kA_\alpha|^2.
\end{align*}
After summing over $k$, the four adverse contributions in
\eqref{eq:differentiated-Simons} are bounded by
$(\gamma_n+6+1+6)SQ=L_nSQ$.
\end{proof}

\markbegin
Set
\[
 D_n=\frac{\sqrt{L_n}+\sqrt3}{\sqrt2}.
\]
\markend

\begin{proposition}\label{prop:gradient}
Every closed minimal immersion considered above satisfies
\begin{equation}\label{eq:gradient-bound}
 \sup_M|\nabla h|\le D_nS_*.
\end{equation}
\end{proposition}

\begin{proof}
From \eqref{eq:Simons-LiLi},
\[
 \frac12\Delta S^2
 =2S|\nabla h|^2+2nS^2-2SR_1+|\nabla S|^2
 \ge2S|\nabla h|^2+S^2(2n-3S).
\]
Fix $c>L_n/2$, let $x_0$ maximize
\[
 F_c=|\nabla h|^2+cS^2,
\]
and put $s=S(x_0)$.  At $x_0$, \Cref{prop:bochner} gives
\[
 0\ge
 \bigl[2n+3+(2c-L_n)s\bigr]|\nabla h|^2(x_0)
 +cs^2(2n-3s).
\]
If $s\le2n/3$, both terms on the right-hand side of the preceding inequality are nonnegative.  Hence $|\nabla h|^2(x_0)=0$; consequently $F_c(x_0)=cs^2\le cS_*^2$.  If $s>2n/3$, then
\[
 |\nabla h|^2(x_0)
 \le\frac{3c}{2c-L_n}s^2,
\]
and hence in both cases
\[
 \sup_M|\nabla h|^2
 \le\sup_MF_c
 \le\left(c+\frac{3c}{2c-L_n}\right)S_*^2.
\]
Write $y=2c-L_n>0$.  The coefficient is
\[
 \frac12\left(y+L_n+3+\frac{3L_n}{y}\right),
\]
which is minimized at $y=\sqrt{3L_n}$.  The minimum is
\[
 \frac{L_n+3+2\sqrt{3L_n}}2
 =\left(\frac{\sqrt{L_n}+\sqrt3}{\sqrt2}\right)^2
 =D_n^2.
\]
This proves \eqref{eq:gradient-bound} and the stated optimality within this auxiliary family.
\end{proof}

\section{The vector-valued sheet defect}\label{sec:vector-defect}

\markbegin
Define only the constants needed in this section by
\[
 E_n=\frac{1+2D_n}{6},
 \qquad
 \sigma_n=\frac1{\sqrt{2n(n+2)}},
 \qquad
 P_n(t)=\sigma_n-E_nt-\frac{t^2}{16n}.
\]
Let $r_n$ be the positive zero of $P_n$, namely
\[
 r_n=\frac{2\sigma_n}
 {E_n+\sqrt{E_n^2+\sigma_n/(4n)}}.
\]
Since $E_n>1$ and $\sigma_n\le1/4$, one has $0<r_n<1/4$.
\markend


\begin{lemma}\label{lem:normal-identities}
For $\varphi_a=\langle f,a\rangle$ and $\xi_a=a^\perp$,
\begin{align}
 \nabla^2\varphi_a(X,Y)
 &=-\varphi_a\langle X,Y\rangle+\langle h(X,Y),\xi_a\rangle,
 \label{eq:Hessian-height}\\
 \nabla_X^\perp\xi_a&=-h(X,\nabla \varphi_a).
 \label{eq:normal-height-derivative}
\end{align}
Let $x\in M$, put $p=f(x)$, and let
$\gamma_v(s)=\exp_x(sv)$ for a unit vector $v\in T_xM$.  If
$T=\gamma_v'$ and $\xi=\xi_p$, then
\begin{equation}\label{eq:normal-second-derivative}
 (\nabla_T^\perp)^2\xi
 =\varphi_p h(T,T)-h(T,A_\xi T)-(\nabla_Th)(T,\nabla \varphi_p).
\end{equation}
At $s=0$,
\[
 \xi=0,
 \qquad
 \nabla_T^\perp\xi=0,
 \qquad
 (\nabla_T^\perp)^2\xi=h_x(v,v).
\]
\end{lemma}

\begin{proof}
Differentiate the Euclidean identity
$a=\varphi_af+\nabla\varphi_a+\xi_a$ in the direction $X$.  The tangent and normal parts of the resulting equation are precisely
\eqref{eq:Hessian-height} and \eqref{eq:normal-height-derivative}.  Along a geodesic, differentiate
$\nabla_T^\perp\xi=-h(T,\nabla\varphi_p)$.  Since $\nabla_TT=0$,
\[
 (\nabla_T^\perp)^2\xi
 =-(\nabla_Th)(T,\nabla\varphi_p)-h(T,\nabla_T\nabla\varphi_p).
\]
Equation \eqref{eq:Hessian-height} says
$\nabla_T\nabla\varphi_p=-\varphi_pT+A_\xi T$, which gives
\eqref{eq:normal-second-derivative}.  At $s=0$ one has
$\varphi_p=1$, $\nabla\varphi_p=0$, and $\xi=0$, proving the initial conditions.
\end{proof}

Before estimating the normal second jet, we record precisely how normal vectors at different points of the geodesic are compared.  For $0\le s_0,s_1<\inj(M)$, the normal connection $\nabla^\perp$ determines an isometry
\[
 P_{s_0\to s_1}^\perp:N_{\gamma_v(s_0)}M\longrightarrow N_{\gamma_v(s_1)}M.
\]
Indeed, if $\eta\in N_{\gamma_v(s_0)}M$, there is a unique normal vector field $V$ along $\gamma_v$ satisfying
\[
 \nabla_T^\perp V=0,
 \qquad
 V(s_0)=\eta,
\]
and we define $P_{s_0\to s_1}^\perp\eta=V(s_1)$.  Since the normal connection is metric, $P_{s_0\to s_1}^\perp$ preserves inner products and norms, and
\[
 (P_{s_0\to s_1}^\perp)^{-1}=P_{s_1\to s_0}^\perp.
\]
If $\zeta(s)$ is any $C^2$ normal vector field along $\gamma_v$, then differentiation in the fixed vector space $N_xM$ gives
\[
 \frac{\dd}{\dd s}\bigl(P_{s\to0}^\perp\zeta(s)\bigr)
 =P_{s\to0}^\perp\nabla_T^\perp\zeta(s),
 \qquad
 \frac{\dd^2}{\dd s^2}\bigl(P_{s\to0}^\perp\zeta(s)\bigr)
 =P_{s\to0}^\perp(\nabla_T^\perp)^2\zeta(s).
\]
Thus normal parallel transport allows us to compare the normal field $\xi_p(\gamma_v(s))$ with the fixed vector $h_x(v,v)\in N_xM$ without choosing a normal frame.

\begin{lemma}\label{lem:vector-jet}
For every $0\le s<\inj(M)$,
\begin{equation}\label{eq:vector-jet}
 \left|
 P_{s\to0}^\perp\xi(\gamma_v(s))
 -\frac12h_x(v,v)s^2
 \right|
 \le E_nS_*s^3+\frac{\sqrt{S_*}}{24}s^4.
\end{equation}
\end{lemma}

\begin{proof}
The spherical distance $d_{\Sph}(p,f(\gamma_v(s)))$ is at most $s$.  Since $\varphi_p=\cos d_{\Sph}(p,f(\gamma_v(s)))$,
\[
 1-\varphi_p\le1-\cos s\le\frac{s^2}{2}.
\]
Moreover, by \eqref{eq:normal-defect-pointwise},
\[
 |\nabla\varphi_p|^2+|\xi|^2
 =1-\varphi_p^2
 =\sin^2 d_{\Sph}(p,f(\gamma_v(s)))
 \le s^2.
\]
Consequently $|\nabla\varphi_p|\le s$ and $|\xi|\le s$.
Because $T=\gamma_v'$ is tangent-parallel along the geodesic and normal parallel transport is an isometry, the global derivative estimate \eqref{eq:gradient-bound} gives
\[
 \left|P_{s\to0}^\perp h_{\gamma_v(s)}(T,T)-h_x(v,v)\right|
 \le\int_0^s|\nabla_T h|\dd\tau
 \le D_nS_*s.
\]
Moreover,
\[
 |h(T,A_\xi T)|\le S_*|\xi|\le S_*s,
 \qquad
 |(\nabla_Th)(T,\nabla\varphi_p)|\le D_nS_*s.
\]
Let
\[
 Y(s)=P_{s\to0}^\perp\xi(\gamma_v(s))\in N_xM.
\]
By the preceding parallel-transport identities and \Cref{lem:normal-identities},
\[
 Y(0)=0,
 \qquad
 Y'(0)=0,
 \qquad
 Y''(0)=h_x(v,v).
\]
Using \eqref{eq:normal-second-derivative}, the preceding three estimates, and $|h(T,T)|\le\sqrt{S_*}$, we obtain
\[
 |Y''(s)-h_x(v,v)|
 \le(2D_n+1)S_*s+\frac{\sqrt{S_*}}2s^2.
\]
Taylor's formula with integral remainder in the fixed Euclidean vector space $N_xM$ gives
\[
 Y(s)-\frac12h_x(v,v)s^2
 =\int_0^s(s-\tau)\bigl(Y''(\tau)-Y''(0)\bigr)\dd\tau.
\]
Therefore
\[
 \left|Y(s)-\frac12h_x(v,v)s^2\right|
 \le\frac{2D_n+1}{6}S_*s^3
 +\frac{\sqrt{S_*}}{24}s^4.
\]
Since $(2D_n+1)/6=E_n$, this is \eqref{eq:vector-jet}.
\end{proof}


\begin{lemma}\label{lem:fourth-moment}
At every $x\in M$,
\[
 \int_{\Sph^{n-1}}|h_x(v,v)|^2\dd\sigma(v)
 =\frac{2\omega_{n-1}}{n(n+2)}S(x).
\]
\end{lemma}

\begin{proof}
The standard spherical fourth-moment identity is
\[
 \int_{\Sph^{n-1}}v_iv_jv_kv_l\dd\sigma
 =\frac{\omega_{n-1}}{n(n+2)}
 (\delta_{ij}\delta_{kl}+\delta_{ik}\delta_{jl}+\delta_{il}\delta_{jk}).
\]
Contract with $h_{ij}^\alpha h_{kl}^\alpha$ and sum over $\alpha$.  The first contraction is the squared mean-curvature vector and vanishes; each of the remaining contractions equals $S(x)$.
\end{proof}

\begin{lemma}\label{lem:angular}
Fix $x\in M$ with $S_x=S(x)>0$, put $p=f(x)$, and set
\[
 t=\frac{S_*s}{\sqrt{S_x}}.
\]
For $0\le t\le r_n$,
\begin{equation}\label{eq:angular-estimate}
 \left(
 \frac1{\omega_{n-1}}
 \int_{\Sph^{n-1}}|\xi_p(\gamma_v(s))|^2\dd\sigma(v)
 \right)^{1/2}
 \ge\sqrt{S_x}\,s^2P_n(t).
\end{equation}
\end{lemma}

\begin{proof}
By \Cref{lem:fourth-moment}, the normalized
$L^2(\Sph^{n-1})$ norm of
$\frac12h_x(v,v)s^2$ is
$\sqrt{S_x}\,s^2\sigma_n$.  Apply the reverse triangle inequality to
\eqref{eq:vector-jet}.  The cubic error divided by
$\sqrt{S_x}\,s^2$ is $E_nt$.  For the quartic error,
\[
 \frac{\sqrt{S_*}s^2}{24\sqrt{S_x}}
 =\frac{t^2\sqrt{S_x}}{24S_*^{3/2}}
 \le\frac{t^2}{24S_*}
 \le\frac{t^2}{16n},
\]
where the last inequality uses \eqref{eq:first-curvature-gap}.  This gives \eqref{eq:angular-estimate}.
\end{proof}

\begin{lemma}\label{lem:radial}
Under the notation of \Cref{lem:angular}, assume $0\le t\le r_n$.  If $J_x(s,v)$ is the polar-coordinate Jacobian centered at $x$, then
\[
 \varphi_p(\exp_x(sv))
 \ge\cos\!\left(\sqrt{\frac{3}{2n}}\,t\right),
\]
\[
 1-\varphi_p(\exp_x(sv))^2\le s^2,
\]
and
\begin{equation}\label{eq:Jacobian-lower}
 J_x(s,v)\ge s^{n-1}
 \left[
 \frac{\sin\!\left(\sqrt{\frac{n+1}{2n}}\,t\right)}
 {\sqrt{\frac{n+1}{2n}}\,t}
 \right]^{n-1}.
\end{equation}
\end{lemma}

\begin{proof}
The ambient spherical distance from $p$ to $f(\exp_x(sv))$ is at most $s$.  Therefore
$\varphi_p\ge\cos s$ and
$1-\varphi_p^2\le s^2$.  Also,
\[
 s=\frac{t\sqrt{S_x}}{S_*}
 \le\frac{t}{\sqrt{S_*}}
 \le\sqrt{\frac{3}{2n}}\,t,
\]
which proves the first two inequalities because $r_n<1/4$.

The comparison is used only inside the normal-coordinate ball considered above.  Indeed, since $t\le r_n<1/4$,
\[
 s=\frac{t\sqrt{S_x}}{S_*}\le\frac{r_n}{\sqrt{S_*}}
 <\frac{1}{4\sqrt{S_*}}
 <\frac{\pi}{2\sqrt{\lambda_nS_*}}
 \le\inj(M),
\]
and, moreover,
\[
 0\le\sqrt{\kappa_nS_*}\,s
 \le\sqrt{\kappa_n}\,r_n<\frac14<\pi.
\]
Thus the model Jacobi factor is positive throughout the interval.  By \eqref{eq:sectional-upper}, G\"unther comparison \cite{Gunther1960} gives
\[
 J_x(s,v)\ge s^{n-1}
 \left(
 \frac{\sin(\sqrt{\kappa_nS_*}\,s)}
 {\sqrt{\kappa_nS_*}\,s}
 \right)^{n-1}.
\]
Furthermore,
\[
 \sqrt{\kappa_nS_*}\,s
 =t\sqrt{\kappa_n\frac{S_x}{S_*}}
 \le\sqrt{\frac{n+1}{2n}}\,t.
\]
The function $\sin z/z$ is decreasing on the interval in question, because $t\le r_n<1/4$.  This proves \eqref{eq:Jacobian-lower}.
\end{proof}


\markbegin
The explicit coefficient appearing in the main theorem is
\begin{equation}\label{eq:epsilon-definition}
 \varepsilon_n=n\int_0^{r_n}tP_n(t)^2
 \cos\!\left(\sqrt{\frac{3}{2n}}\,t\right)
 \left[
 \frac{\sin\!\left(\sqrt{\frac{n+1}{2n}}\,t\right)}
 {\sqrt{\frac{n+1}{2n}}\,t}
 \right]^{n-1}\dd t.
\end{equation}
The sine quotient is interpreted as $1$ at the origin.
\markend

\begin{theorem}\label{thm:defect-weighted}
If $p\in f(M)$ and
$f^{-1}(p)=\{x_1,\ldots,x_m\}$, then
\begin{equation}\label{eq:defect-weighted}
 \mathcal R_p\ge\beta_n\varepsilon_n
 \sum_{j=1}^m
 \left(\frac{S(x_j)}{S_*}\right)^2.
\end{equation}
\end{theorem}

\begin{proof}
Put $S_j=S(x_j)$.  A sheet with $S_j=0$ contributes zero to the right side and may be omitted.  If $S_j>0$, set
\[
 \varrho_j=r_n\frac{\sqrt{S_j}}{S_*}.
\]
Since $r_n<1/4$,
\[
 \varrho_j\le\frac{1}{4\sqrt{S_*}}.
\]
The injectivity and separation estimates in \Cref{lem:scale-separation} imply that the geodesic balls $B_{\varrho_j}(x_j)$ lie in their normal-coordinate neighborhoods and are pairwise disjoint.  Indeed,
\[
 \frac{1}{4\sqrt{S_*}}
 <\frac{\pi}{2\sqrt{\lambda_nS_*}}
\]
and twice the left side is smaller than the separation bound
\eqref{eq:sheet-separation}.

Fix one sheet $x=x_j$.  On $B_{\varrho_j}(x)$, use geodesic polar coordinates and combine \Cref{lem:angular,lem:radial}.  Since
$(1-\varphi_p^2)^{-n/2-1}\ge s^{-n-2}$,
\begin{align*}
 &\int_{B_{\varrho_j}(x)}
 \frac{\varphi_p|p^\perp|^2}
 {(1-\varphi_p^2)^{n/2+1}}\dd V\\
 &\quad\ge
 \omega_{n-1}S_j\int_0^{\varrho_j}s
 P_n\!\left(\frac{S_*s}{\sqrt{S_j}}\right)^2
 \cos\!\left(
 \sqrt{\frac{3}{2n}}\frac{S_*s}{\sqrt{S_j}}
 \right)
 \left[
 \frac{\sin\!\left(
 \sqrt{\frac{n+1}{2n}}\frac{S_*s}{\sqrt{S_j}}
 \right)}
 {\sqrt{\frac{n+1}{2n}}\frac{S_*s}{\sqrt{S_j}}}
 \right]^{n-1}\dd s.
\end{align*}
Change variables $t=S_*s/\sqrt{S_j}$.  By the definition
\eqref{eq:epsilon-definition}, the last expression equals
\[
 \frac{\omega_{n-1}}n\varepsilon_n
 \left(\frac{S_j}{S_*}\right)^2
 =\beta_n\varepsilon_n
 \left(\frac{S_j}{S_*}\right)^2.
\]
Summing over the disjoint balls proves \eqref{eq:defect-weighted}.
\end{proof}

\begin{proof}[\textbf{Proof of \Cref{thm:main}}]
Combine \Cref{thm:two-remainders,thm:defect-weighted} and discard the nonnegative global normal-energy term.
\end{proof}

\section{Linearly full volume bounds}\label{sec:full}

In this section the immersion is linearly full.  Put
\[
 N=n+q+1.
\]

\subsection{Integrated normal energy and quantitative fullness}

Define two symmetric endomorphisms of $\R^N$ by
\[
 G=\int_M f\otimes f\dd V,
 \qquad
 C=\int_M P_{N_xM}\dd V_x,
\]
where $P_{N_xM}$ is the orthogonal projection onto the normal space of $M$ inside the sphere.

\begin{lemma}\label{lem:integrated-normal}
The endomorphisms $G$ and $C$ satisfy
\begin{equation}\label{eq:integrated-normal}
 C=\Vol(M)I-(n+1)G,
 \qquad
 \tr C=q\Vol(M).
\end{equation}
For every $a\in\R^N$,
\[
 \int_M|a^\perp|^2\dd V=\langle Ca,a\rangle.
\]
\end{lemma}

\begin{proof}
At every point of $M$, the Euclidean orthogonal decomposition is
\[
 I=f\otimes f+P_{T_xM}+P_{N_xM}.
\]
For $a,b\in\R^N$, Takahashi's identity and integration by parts give
\begin{align*}
 \int_M\langle a^T,b^T\rangle\dd V
 &=\int_M\langle\nabla\varphi_a,\nabla\varphi_b\rangle\dd V\\
 &=n\int_M\varphi_a\varphi_b\dd V
 =\langle nGa,b\rangle.
\end{align*}
Thus $\int_MP_{T_xM}\dd V=nG$.  Integrating the pointwise orthogonal decomposition proves the first formula in \eqref{eq:integrated-normal}; taking traces proves the second.  The final identity follows directly from the definition of $C$.
\end{proof}

Let $p_1,\ldots,p_L\in f(M)$ and let
$w_1,\ldots,w_L>0$ with $\sum_{\ell=1}^Lw_\ell=1$.  Define the image-frame operator
\[
 \mathscr P_\mu=\sum_{\ell=1}^Lw_\ell p_\ell\otimes p_\ell,
 \qquad
 \tau_\mu=\lambda_{\min}(\mathscr P_\mu).
\]
For an image point $p$, write
\[
 m(p)=\#f^{-1}(p),
 \qquad
 \mathcal K(p)=\sum_{x\in f^{-1}(p)}
 \left(\frac{S(x)}{S_*}\right)^2,
\]
and define the weighted averages
\[
 \overline m_\mu=\sum_{\ell=1}^Lw_\ell m(p_\ell),
 \qquad
 \overline{\mathcal K}_\mu
 =\sum_{\ell=1}^Lw_\ell\mathcal K(p_\ell).
\]

\begin{theorem}\label{thm:frame-full}
If $\tau_\mu>0$, then
\begin{equation}\label{eq:frame-full}
 \frac{\Vol(M)}{\omega_n}
 \ge
 \frac{\overline m_\mu+
 \varepsilon_n\overline{\mathcal K}_\mu}
 {1-q\tau_\mu/(n+2)}.
\end{equation}
The denominator is positive.  A linearly full immersion admits a finite image frame for which $\tau_\mu>0$.
\end{theorem}

\begin{proof}
Apply \Cref{thm:two-remainders} to every $p_\ell$, multiply by $w_\ell$, and sum.  By \Cref{thm:defect-weighted},
\[
 \mathcal R_{p_\ell}\ge
 \beta_n\varepsilon_n\mathcal K(p_\ell).
\]
Furthermore, \Cref{lem:integrated-normal} and the positivity of $C$ give
\begin{align*}
 \sum_{\ell=1}^Lw_\ell
 \int_M|p_\ell^\perp|^2\dd V
 &=\tr(\mathscr P_\mu C)\\
 &\ge\tau_\mu\tr C
 =q\tau_\mu\Vol(M).
\end{align*}
Thus
\[
 \Vol(M)\ge
 \omega_n\bigl(\overline m_\mu+
 \varepsilon_n\overline{\mathcal K}_\mu\bigr)
 +\frac{q\tau_\mu}{n+2}\Vol(M),
\]
which is \eqref{eq:frame-full}.  Since $\tr \mathscr P_\mu=1$,
$\tau_\mu\le1/N$, and hence
$q\tau_\mu/(n+2)<1$.  If the immersion is linearly full, choose $N$ image points spanning $\R^N$ and give them positive weights.  Then $\mathscr P_\mu$ is positive definite.
\end{proof}

\begin{remark}
Linear fullness alone guarantees a frame with $\tau_\mu>0$, but not a uniform lower bound for $\tau_\mu$ in terms of $n$ and $q$.  The next argument uses the entire height-function space and gives a universal codimension-sensitive estimate.
\end{remark}

\markbegin
\subsection{A mean-value estimate on the Euclidean minimal cone}

Let
\[
 \widehat M=(0,\infty)\times M,
 \qquad
 F(r,x)=rf(x)\in\R^N.
\]
With the cone metric $\widehat g=\dd r^2+r^2g$, the map $F$ is a minimal immersion of dimension $n+1$.  For
$\varphi_a(x)=\langle f(x),a\rangle$, define
\[
 \widehat{\varphi}_a(r,x)=r \varphi_a(x)=\langle F(r,x),a\rangle.
\]
The cone Laplacian is
\[
 \Delta_{\widehat M}
 =\partial_r^2+\frac nr\partial_r+\frac1{r^2}\Delta_M.
\]
Since $\Delta_M\varphi_a=-n\varphi_a$,
\[
 \Delta_{\widehat M}\widehat{\varphi}_a=0,
 \qquad
 \Delta_{\widehat M}\widehat{\varphi}_a^2
 =2|\widehat\nabla \widehat{\varphi}_a|^2\ge0.
\]

The following identity is the smooth minimal-immersion form of the classical mean-value inequalities of Michael--Simon \cite{MichaelSimon1973}.  We include its exact proof because the cone is immersed and the optimizing balls may cross the vertex.

\begin{lemma}[Michael--Simon \cite{MichaelSimon1973}]\label{lem:mean-value-general}
Let $F:\Sigma^d\looparrowright\R^N$ be minimal, let $z\in\R^N$, put
$R=|F-z|$, and let $\Phi\ge0$ be a smooth subharmonic function.  Let $J\subset(0,\infty)$ be an interval such that $\{R<s\}$ has compact closure in $\Sigma$ for every $s\in J$.  Then, at every regular value $s\in J$,
\begin{align}\label{eq:mean-value-monotonicity}
 \frac{\dd}{\dd s}
 \left(s^{-d}\int_{\{R<s\}}\Phi\dd V\right)
 ={}&s^{-d-1}\int_{\{R=s\}}
 \frac{\Phi|(F-z)^\perp|^2}{s|\nabla R|}\dd\sigma\notag\\
 &+\frac12s^{-d-1}
 \int_{\{R<s\}}(s^2-R^2)\Delta\Phi\dd V.
\end{align}
In particular, the normalized integral is nondecreasing on $J$.
\end{lemma}

\begin{proof}
Let $\Omega_s=\{R<s\}$ and
$I(s)=\int_{\Omega_s}\Phi\dd V$.  Minimality gives
\[
 \Delta R^2=2d.
\]
Apply Green's second identity on $\Omega_s$ to $\Phi$ and
$\psi=(s^2-R^2)/2$.  Since $\psi=0$ on $\partial\Omega_s$,
$\Delta\psi=-d$, and
$\partial_\nu\psi=-s|\nabla R|$, one obtains
\[
 dI(s)=s\int_{\{R=s\}}\Phi|\nabla R|\dd\sigma
 -\frac12\int_{\Omega_s}(s^2-R^2)\Delta\Phi\dd V.
\]
Coarea gives
\[
 I'(s)=\int_{\{R=s\}}\frac{\Phi}{|\nabla R|}\dd\sigma.
\]
Therefore
\begin{align*}
 \frac{\dd}{\dd s}(s^{-d}I(s))
 =s^{-d-1}\biggl[&s\int_{\{R=s\}}
 \Phi\left(\frac1{|\nabla R|}-|\nabla R|\right)\dd\sigma\\
 &+\frac12\int_{\Omega_s}(s^2-R^2)\Delta\Phi\dd V\biggr].
\end{align*}
On $\{R=s\}$,
\[
 |\nabla R|^2=\frac{|(F-z)^T|^2}{s^2}
 =1-\frac{|(F-z)^\perp|^2}{s^2}.
\]
Substitution gives \eqref{eq:mean-value-monotonicity}.  The calculation was made at regular values of $R$.  By the coarea formula, the normalized integral is locally absolutely continuous on $J$; approximating any radius in $J$ from above and below by regular values extends the monotonicity statement to all of $J$.
\end{proof}

\begin{lemma}\label{lem:cone-mean}
For every $a\in\R^N$ and every $\rho>0$,
\begin{equation}\label{eq:cone-Linf-rho}
 \|\varphi_a\|_{L^\infty(M)}^2
 \le
 \frac{(1+\rho)^{n+3}}
 {(n+3)\beta_{n+1}\rho^{n+1}}
 \int_M\varphi_a^2\dd V.
\end{equation}
Consequently,
\begin{equation}\label{eq:cone-Linf}
 \|\varphi_a\|_{L^\infty(M)}^2
 \le\Lambda_n\int_M\varphi_a^2\dd V,
 \qquad
 \Lambda_n=
 \frac{(n+3)^{n+2}}
 {4\beta_{n+1}(n+1)^{n+1}}.
\end{equation}
\end{lemma}

\begin{proof}
Choose $x_0\in M$ with
$|\varphi_a(x_0)|=\|\varphi_a\|_\infty$ and put $z=f(x_0)$.  Repeat the proof of \Cref{lem:mean-value-general} on the truncated cone $\{r>\epsilon\}$ with $d=n+1$ and $\Phi=\widehat{\varphi}_a^2$.  For fixed $\epsilon>0$, the relevant extrinsic sublevel set is relatively compact because $M$ is compact and $r$ is bounded above on an extrinsic ball.  The truncation introduces an inner boundary $\{r=\epsilon\}$.  On this boundary,
\[
 \widehat{\varphi}_a^2=O(\epsilon^2),
 \qquad
 |\partial_r(\widehat{\varphi}_a^2)|=O(\epsilon),
 \qquad
 \dd\sigma_\epsilon=\epsilon^n\dd V_M.
\]
In Green's identity the two extra inner-boundary terms have the form
\[
 \int_{\{r=\epsilon\}}\widehat{\varphi}_a^2\,\partial_\nu\psi\,\dd\sigma_\epsilon,
 \qquad
 -\int_{\{r=\epsilon\}}\psi\,\partial_\nu(\widehat{\varphi}_a^2)\,\dd\sigma_\epsilon,
\]
where $\psi=(\rho^2-R^2)/2$.  Since $\psi$ and $\partial_\nu\psi$ are uniformly bounded as $\epsilon\downarrow0$, these two terms are respectively $O(\epsilon^{n+2})$ and $O(\epsilon^{n+1})$, and both vanish.  Thus letting $\epsilon\downarrow0$ proves the mean-value monotonicity formula across the cone vertex.

At the point $(1,x_0)$, the small extrinsic-ball density of the immersed cone is a positive integer, equal to the number of cone sheets through $z$.  Hence
\[
 \widehat{\varphi}_a(1,x_0)^2
 \le\frac1{\beta_{n+1}\rho^{n+1}}
 \int_{\{|F-z|<\rho\}}\widehat{\varphi}_a^2\dd V_{\widehat M}.
\]
This inequality remains valid when several sheets pass through $z$, because the small-radius limit is then the multiplicity times
$\widehat{\varphi}_a(1,x_0)^2$.

If $|rf(x)-z|<\rho$, the reverse triangle inequality gives
$|r-1|<\rho$, and in particular $0<r<1+\rho$.  Since
$\dd V_{\widehat M}=r^n\dd r\dd V_M$,
\begin{align*}
 \int_{\{|F-z|<\rho\}}\widehat{\varphi}_a^2\dd V_{\widehat M}
 &\le\int_M\int_0^{1+\rho}
 r^{n+2}\varphi_a(x)^2\dd r\dd V\\
 &=\frac{(1+\rho)^{n+3}}{n+3}
 \int_M\varphi_a^2\dd V.
\end{align*}
This proves \eqref{eq:cone-Linf-rho}.  The logarithmic derivative of
$(1+\rho)^{n+3}/\rho^{n+1}$ vanishes only at
$\rho=(n+1)/2$, which is the global minimum.  Substitution gives
\eqref{eq:cone-Linf}.
\end{proof}

\begin{theorem}\label{thm:cone-full}
If the immersion is linearly full, then
\begin{equation}\label{eq:cone-full}
 \frac{\Vol(M)}{\omega_n}
 \ge b_n(n+q+1),
 \qquad
 b_n=\frac{4(n+1)^n}{(n+3)^{n+2}}.
\end{equation}
Moreover,
\[
 \frac{\Vol(M)}{\omega_n}>
 \frac{n+q+1}{2(n+3)^2}.
\]
\end{theorem}

\begin{proof}
The space of height functions
\[
 \mathcal H_f=\{\varphi_a=\langle f,a\rangle:a\in\R^N\}
\]
has dimension $N$ by linear fullness.  Choose an
$L^2(M)$-orthonormal basis $\varphi_1,\ldots,\varphi_N$ and put
\[
 K(x)=\sum_{A=1}^N\varphi_A(x)^2.
\]
For fixed $x$, the norm of the evaluation functional on $\mathcal H_f$ is $K(x)^{1/2}$.  Thus there is a function
$\varphi\in\mathcal H_f$ with $\|\varphi\|_2=1$ and $\varphi(x)^2=K(x)$.  By
\Cref{lem:cone-mean}, $K(x)\le\Lambda_n$.  Integrating gives
\[
 N=\int_MK\dd V\le\Lambda_n\Vol(M).
\]
Since $\omega_n=(n+1)\beta_{n+1}$, this is
\eqref{eq:cone-full}.

Finally,
\[
 \left(\frac{n+1}{n+3}\right)^n
 =\left(1+\frac2{n+1}\right)^{-n}
 >e^{-2}>\frac18.
\]
Therefore $b_n>1/[2(n+3)^2]$, proving the last assertion.
\end{proof}

\begin{proof}[\textbf{Proof of \Cref{thm:full-main}}]
Apply \Cref{thm:main} at $p_*=f(x_*)$.  Since
\[
 m_*+\varepsilon_n\mathcal K_*
 =1+\varepsilon_n
 +(m_*-1)+\varepsilon_n(\mathcal K_*-1),
\]
we obtain
\[
 \frac{\Vol(M)}{\omega_n}
 \ge m_*+\varepsilon_n\mathcal K_*.
\]
On the other hand, \Cref{thm:cone-full} gives
\[
 \frac{\Vol(M)}{\omega_n}\ge b_n(n+q+1).
\]
Taking the larger of these two independent lower bounds proves \eqref{eq:full-max}.
\end{proof}

\begin{proof}[\textbf{Proof of \Cref{cor:full-additive}}]
Set
\[
 A=m_*+\varepsilon_n\mathcal K_*,
 \qquad
 B=b_n(n+q+1).
\]
The number $(n+1)b_n$ is smaller than $4/(n+3)$ and hence is less than $1$.  Define
\[
 \theta_n=
 \frac{\varepsilon_n}
 {1+\varepsilon_n-(n+1)b_n}.
\]
Then $0<\theta_n<1$, $1-\theta_n=\rho_n$, and
$\theta_nb_n=\delta_n$, where $\rho_n$ and $\delta_n$ are given in
\eqref{eq:full-constants}.  Since \Cref{thm:full-main} gives both
$\Vol(M)/\omega_n\ge A$ and $\Vol(M)/\omega_n\ge B$, it also gives their convex combination.  Using
\[
 (1-\theta_n)(1+\varepsilon_n)
 +\theta_nb_n(n+1)=1
\]
yields exactly \eqref{eq:full-additive}.  The explicit estimates for $\delta_n$ are established in \Cref{prop:delta-full}.
\end{proof}

\begin{proposition}\label{prop:delta-full}
The coefficient $\delta_n$ in \eqref{eq:full-constants} satisfies
\begin{equation}\label{eq:delta-polynomial}
 \delta_n>
 \frac{1}{440n(n+2)^2(n+3)^2}
 >\frac{1}{11000n^5}.
\end{equation}
Moreover,
\[
 \delta_n\sim0.004981045225\,n^{-5}.
\]
\end{proposition}

\begin{proof}
By \Cref{prop:epsilon-polynomial},
\[
 \varepsilon_n>\frac1{110n(n+2)^2}.
\]
As shown in the proof of \Cref{thm:cone-full},
$b_n>1/[2(n+3)^2]$.  Also $\varepsilon_n<1$, so
\[
 1+\varepsilon_n-(n+1)b_n<2.
\]
Substitution into \eqref{eq:full-constants} gives the first lower bound in \eqref{eq:delta-polynomial}.  Since
$n+2\le2n$ and $n+3\le5n/2$ for $n\ge2$, the second follows.

By \Cref{prop:epsilon-asymptotic},
$\varepsilon_n\sim C_\infty n^{-3}$ with
$C_\infty\approx0.00920130564911321$.  Also
$b_n\sim4e^{-2}n^{-2}$, while the denominator in
\eqref{eq:full-constants} tends to $1$.  Therefore
\[
 \delta_n\sim4e^{-2}C_\infty n^{-5}
 \approx0.004981045225\,n^{-5}.
\]
\end{proof}

\begin{remark}
The factor $n+q+1$ in the second branch of \Cref{thm:full-main} comes from the dimension of the full height-function space and the mean-value monotonicity of the one-homogeneous extensions on the minimal cone.  The $q$-term in \Cref{cor:full-additive} is its normalized convex interpolation with the curvature--multiplicity branch.  No heat-kernel estimate is used.
\end{remark}

\markend

\markbegin
\section{Quantitative hyperplane partitions}\label{sec:hyperplane-partition}

We prove \Cref{thm:hyperplane-partition} by localizing the minimal-cone mean-value argument to a single nodal domain.  The only delicate point is that the squared height function is extended by zero across the lateral boundary of that nodal domain.  We therefore first verify its weak subharmonicity in full detail.

\begin{lemma}\label{lem:weak-zero-extension}
Let $(X,g_X)$ be a smooth Riemannian manifold, let $v\in C^\infty(X)$ satisfy $\Delta_Xv=0$, and let $\mathcal O$ be a connected component of $\{v\ne0\}$.  Define
\[
 \Phi_{\mathcal O}=
 \begin{cases}
 v^2,&\text{on }\mathcal O,\\
 0,&\text{on }X\setminus\mathcal O.
 \end{cases}
\]
Then $\Phi_{\mathcal O}\in W^{1,2}_{\mathrm{loc}}(X)\cap C^0(X)$, its weak gradient is $2v\nabla v$ on $\mathcal O$ and zero almost everywhere on the complement, and for every nonnegative $\zeta\in C_c^\infty(X)$,
\[
 -\int_X\langle\nabla\Phi_{\mathcal O},\nabla\zeta\rangle\dd V_X
 =2\int_{\mathcal O}|\nabla v|^2\zeta\dd V_X\ge0.
\]
In particular, $\Phi_{\mathcal O}$ is weakly subharmonic and its distributional Laplacian is the nonnegative Radon measure
\[
 \Delta_X\Phi_{\mathcal O}
 =2|\nabla v|^2\mathbf 1_{\mathcal O}\dd V_X.
\]
\end{lemma}

\begin{proof}
Replacing $v$ by $-v$ if necessary, assume that $v>0$ on $\mathcal O$.  The continuity of $v$ implies $v=0$ on $\partial\mathcal O$.  Let $\epsilon>0$ be a regular value of $v$ and define
\[
 v_{\epsilon,\mathcal O}(x)=
 \begin{cases}
 v(x)-\epsilon,&x\in\mathcal O\text{ and }v(x)>\epsilon,\\
 0,&\text{otherwise}.
 \end{cases}
\]
On every compact subset of $X$, the support of $v_{\epsilon,\mathcal O}$ stays a positive distance from $\partial\mathcal O$, because $v$ vanishes there.  Thus $v_{\epsilon,\mathcal O}$ belongs to $W^{1,2}_{\mathrm{loc}}(X)$ and
\[
 \nabla v_{\epsilon,\mathcal O}
 =\mathbf 1_{\mathcal O\cap\{v>\epsilon\}}\nabla v
 \quad\text{almost everywhere}.
\]
Set $\Phi_{\epsilon}=v_{\epsilon,\mathcal O}^2$.  The boundary value and the normal derivative of $(v-\epsilon)^2$ both vanish on the regular hypersurface $\{v=\epsilon\}\cap\mathcal O$.  Integration by parts on $\mathcal O\cap\{v>\epsilon\}$ therefore gives, for every nonnegative $\zeta\in C_c^\infty(X)$,
\[
 -\int_X\langle\nabla\Phi_{\epsilon},\nabla\zeta\rangle\dd V_X
 =2\int_{\mathcal O\cap\{v>\epsilon\}}|\nabla v|^2\zeta\dd V_X,
\]
because $\Delta_Xv=0$.  Choose regular values $\epsilon_j\downarrow0$.  Dominated convergence gives
\[
 \Phi_{\epsilon_j}\longrightarrow\Phi_{\mathcal O}
 \quad\text{and}\quad
 \nabla\Phi_{\epsilon_j}\longrightarrow
 2v\nabla v\,\mathbf 1_{\mathcal O}
 \quad\text{in }L^2_{\mathrm{loc}}.
\]
Passing to the limit proves the asserted weak identity.  A positive distribution is a Radon measure, and the displayed identity identifies it with
$2|\nabla v|^2\mathbf 1_{\mathcal O}\dd V_X$.
\end{proof}

\begin{lemma}\label{lem:nodal-cone-estimate}
Let $a\in\Sph^{n+q}$ and let $\Omega$ be a connected component of $\{\varphi_a\ne0\}$.  Then
\[
 \|\varphi_a\|_{L^\infty(\Omega)}^2
 \le
 \Lambda_n\int_{\Omega}\varphi_a^2\dd V,
 \qquad
 \Lambda_n=
 \frac{(n+3)^{n+2}}
 {4\beta_{n+1}(n+1)^{n+1}}.
\]
\end{lemma}

\begin{proof}
Use the Euclidean minimal cone
\[
 \widehat M=(0,\infty)\times M,
 \qquad
 F(r,x)=rf(x),
 \qquad
 \widehat\varphi_a(r,x)=r\varphi_a(x).
\]
As in \Cref{sec:full}, the cone has dimension $d=n+1$ and
$\Delta_{\widehat M}\widehat\varphi_a=0$.  The set
$\widehat\Omega=(0,\infty)\times\Omega$ is a connected component of
$\{\widehat\varphi_a\ne0\}$.  Extend the square of the homogeneous height function by zero:
\[
 \Phi_\Omega(r,x)=
 \begin{cases}
 \widehat\varphi_a(r,x)^2,&x\in\Omega,\\
 0,&x\notin\Omega.
 \end{cases}
\]
By \Cref{lem:weak-zero-extension}, $\Phi_\Omega$ is nonnegative and weakly subharmonic on the smooth cone $(0,\infty)\times M$, and
\[
 \Delta_{\widehat M}\Phi_\Omega
 =2|\widehat\nabla\widehat\varphi_a|^2
 \mathbf 1_{\widehat\Omega}\dd V_{\widehat M}.
\]

We next explain why the mean-value monotonicity in
\Cref{lem:mean-value-general} remains valid for this weakly subharmonic function.  Write
$\mu_\Omega=\Delta_{\widehat M}\Phi_\Omega$, fix $z\in\R^{n+q+1}$, put $R=|F-z|$, and let
\[
 I_\epsilon(s)=
 \int_{\{R<s\}\cap\{r>\epsilon\}}
 \Phi_\Omega\dd V_{\widehat M}.
\]
We first justify the Green identity used below.  Choose regular values $\delta_j\downarrow0$ and apply the calculation to the truncations $\Phi_{\Omega,\delta_j}$ constructed in the proof of \Cref{lem:weak-zero-extension}.  These functions are $C^1$ and piecewise smooth; both their value and their normal derivative vanish on the additional free boundary.  Hence no lateral boundary term is produced.  As $j\to\infty$, the truncations converge to $\Phi_\Omega$ in $W^{1,2}_{\mathrm{loc}}$, while their distributional Laplacians converge weakly as Radon measures to $\mu_\Omega$.  Thus the resulting identities pass to the limit and give the distributional Green identity for $\Phi_\Omega$.

For a regular radius $s$, use the nonnegative test function
$\psi_s=(s^2-R^2)/2$ on $\{R<s\}\cap\{r>\epsilon\}$.  Since
$\Delta R^2=2d$, the distributional Green identity gives
\[
 \begin{aligned}
 dI_\epsilon(s)
 ={}&s\int_{\{R=s\}\cap\{r>\epsilon\}}
 \Phi_\Omega|\widehat\nabla R|\dd\sigma
 -\int_{\{R<s\}\cap\{r>\epsilon\}}
 \frac{s^2-R^2}{2}\dd\mu_\Omega
 +\mathcal E_\epsilon(s),
 \end{aligned}
\]
where $\mathcal E_\epsilon(s)$ is the sum of the two boundary terms on
$\{r=\epsilon\}$.  Coarea also gives
\[
 I_\epsilon'(s)=
 \int_{\{R=s\}\cap\{r>\epsilon\}}
 \frac{\Phi_\Omega}{|\widehat\nabla R|}\dd\sigma
\]
for almost every regular $s$.  Subtracting the preceding two identities and using
\[
 1-|\widehat\nabla R|^2
 =\frac{|(F-z)^\perp|^2}{s^2}
 \quad\text{on }\{R=s\}
\]
yields
\[
 \begin{aligned}
 \frac{\dd}{\dd s}\bigl(s^{-d}I_\epsilon(s)\bigr)
 ={}&s^{-d-1}\int_{\{R=s\}\cap\{r>\epsilon\}}
 \frac{\Phi_\Omega|(F-z)^\perp|^2}
 {s|\widehat\nabla R|}\dd\sigma\\
 &+s^{-d-1}\int_{\{R<s\}\cap\{r>\epsilon\}}
 \frac{s^2-R^2}{2}\dd\mu_\Omega
 -s^{-d-1}\mathcal E_\epsilon(s).
 \end{aligned}
\]
The first two terms on the right are nonnegative.  Moreover,
\[
 0\le\Phi_\Omega\le r^2,
 \qquad
 |\widehat\nabla\Phi_\Omega|
 \le2r
 \quad\text{almost everywhere},
\]
because $|\varphi_a|\le1$ and
$|\widehat\nabla\widehat\varphi_a|^2
 =\varphi_a^2+|\nabla\varphi_a|^2\le1$.
The two summands in $\mathcal E_\epsilon(s)$ contain, respectively,
$\Phi_\Omega\,\partial_\nu\psi_s$ and
$\psi_s\,\partial_\nu\Phi_\Omega$.  Since the measure of
$\{r=\epsilon\}$ is $\epsilon^n\Vol(M)$ and
$\psi_s,\partial_\nu\psi_s$ are uniformly bounded for fixed $s$, these summands are
$O(\epsilon^{n+2})$ and $O(\epsilon^{n+1})$.  Hence
$\mathcal E_\epsilon(s)\to0$.  Letting $\epsilon\downarrow0$ gives
\[
 \begin{aligned}
 \frac{\dd}{\dd s}
 \left(s^{-d}\int_{\{R<s\}}\Phi_\Omega\dd V_{\widehat M}\right)
 ={}&s^{-d-1}\int_{\{R=s\}}
 \frac{\Phi_\Omega|(F-z)^\perp|^2}
 {s|\widehat\nabla R|}\dd\sigma\\
 &+s^{-d-1}\int_{\{R<s\}}
 \frac{s^2-R^2}{2}\dd\mu_\Omega\ge0
 \end{aligned}
\]
for almost every regular $s$.  Local absolute continuity, obtained from coarea, extends the monotonicity to all $s>0$.

Since $\overline\Omega$ is compact in $M$, the continuous function $|\varphi_a|$ attains its maximum on $\overline\Omega$.  Moreover, $\varphi_a=0$ on $\partial\Omega$ and $\varphi_a$ is nonzero on $\Omega$, so this maximum is positive and is attained at an interior point $x_\Omega\in\Omega$.  Set
\[
 \alpha_\Omega
 =|\varphi_a(x_\Omega)|
 =\|\varphi_a\|_{L^\infty(\Omega)}>0,
 \qquad
 z=f(x_\Omega).
\]  Since $f$ is an immersion and $M$ is compact, the set
$f^{-1}(z)$ is finite.  The preimage of $z$ under the cone immersion is
$\{1\}\times f^{-1}(z)$.  Choose pairwise disjoint coordinate neighborhoods of these preimages on which $F$ is an embedding.  In normal coordinates on each such sheet,
$F=z+\dd F_0(\cdot)+O(|\cdot|^2)$, the volume density is
$1+O(|\cdot|^2)$, and $\Phi_\Omega$ converges to its value at the center.  The contribution of that sheet to the extrinsic ball is therefore
\[
 \beta_{n+1}s^{n+1}\Phi_\Omega(1,y)+o(s^{n+1}).
\]
For sufficiently small $s$, no point outside these neighborhoods maps into
$B_s(z)$.  Summing the local contributions gives
\[
 \lim_{s\downarrow0}
 \frac{1}{\beta_{n+1}s^{n+1}}
 \int_{\{|F-z|<s\}}\Phi_\Omega\dd V_{\widehat M}
 =\sum_{y\in f^{-1}(z)}\Phi_\Omega(1,y)
 \ge\alpha_\Omega^2,
\]
because the summand corresponding to $y=x_\Omega$ equals
$\alpha_\Omega^2$ and all other summands are nonnegative.  The preceding monotonicity therefore implies, for every $\rho>0$,
\[
 \alpha_\Omega^2
 \le
 \frac{1}{\beta_{n+1}\rho^{n+1}}
 \int_{\{|F-z|<\rho\}}\Phi_\Omega\dd V_{\widehat M}.
\]
If $|rf(x)-z|<\rho$, then $|r-1|<\rho$, hence
$0<r<1+\rho$.  Since
$\dd V_{\widehat M}=r^n\dd r\dd V_M$ and
$\Phi_\Omega=r^2\varphi_a^2$ on $\widehat\Omega$, we obtain
\[
 \begin{aligned}
 \int_{\{|F-z|<\rho\}}\Phi_\Omega\dd V_{\widehat M}
 &\le
 \int_\Omega\int_0^{1+\rho}
 r^{n+2}\varphi_a(x)^2\dd r\dd V\\
 &=\frac{(1+\rho)^{n+3}}{n+3}
 \int_\Omega\varphi_a^2\dd V.
 \end{aligned}
\]
Thus
\[
 \alpha_\Omega^2
 \le
 \frac{(1+\rho)^{n+3}}
 {(n+3)\beta_{n+1}\rho^{n+1}}
 \int_\Omega\varphi_a^2\dd V.
\]
The logarithmic derivative of
$(1+\rho)^{n+3}/\rho^{n+1}$ vanishes only at
$\rho=(n+1)/2$, where the global minimum is attained.  Substitution gives the stated value of $\Lambda_n$.
\end{proof}

\begin{proof}[\textbf{Proof of \Cref{thm:hyperplane-partition}}]
If $\varphi_a\equiv0$, then $f(M)\subset\mathcal H_a$ and
$\mathcal N_a=0$.  Assume henceforth that $\varphi_a$ is nontrivial and let
$\Omega$ be one of its nodal domains.  By
\Cref{lem:nodal-cone-estimate},
\[
 \alpha_\Omega^2
 \le\Lambda_n\int_\Omega\varphi_a^2\dd V.
\]
On the other hand,
\[
 \int_\Omega\varphi_a^2\dd V
 \le\alpha_\Omega^2\Vol(\Omega).
\]
Since $\alpha_\Omega>0$, cancellation yields
\[
 \Vol(\Omega)\ge\Lambda_n^{-1}.
\]
Using $\omega_n=(n+1)\beta_{n+1}$, we have
\[
 \Lambda_n^{-1}
 =\frac{4\beta_{n+1}(n+1)^{n+1}}{(n+3)^{n+2}}
 =\frac{4(n+1)^n}{(n+3)^{n+2}}\omega_n
 =b_n\omega_n.
\]
This proves the asserted lower bound for every nodal domain.

The nodal domains are pairwise disjoint.  Hence any collection of $k$ distinct nodal domains has total volume at least
$k b_n\omega_n$ and at most $\Vol(M)$.  It follows first that the number of nodal domains is finite and then that
\[
 \mathcal N_a
 \le b_n^{-1}\frac{\Vol(M)}{\omega_n}
 =\frac{(n+3)^{n+2}}{4(n+1)^n}
 \frac{\Vol(M)}{\omega_n}.
\]
Finally,
\[
 \frac1{b_n}
 =\frac{(n+3)^2}{4}
 \left(1+\frac2{n+1}\right)^n
 <\frac{e^2}{4}(n+3)^2,
\]
because $(1+x)^n<e^{nx}$ and $2n/(n+1)<2$.  This proves the final estimate.
\end{proof}
\markend

\appendix
\section{Explicit constants and asymptotics}\label{sec:constants}

\begin{proposition}\label{prop:epsilon-polynomial}
For every $n\ge2$,
\[
 \frac{1}{110n(n+2)^2}<\varepsilon_n<
 \frac{1}{104n(n+2)^2}.
\]
Consequently,
\[
 \varepsilon_n>\frac{1}{440n^3}.
\]
\end{proposition}

\begin{proof}
For the upper bound, note that
$P_n(t)\le\sigma_n-E_nt$, both geometric factors in
\eqref{eq:epsilon-definition} are at most $1$, and
$r_n<\sigma_n/E_n$.  Therefore
\[
 \varepsilon_n<
 n\int_0^{\sigma_n/E_n}t(\sigma_n-E_nt)^2\dd t
 =\frac{1}{48E_n^2n(n+2)^2}.
\]
The smallest value of $L_n$ occurs at $n=2$ and equals $29/2$.  Hence
\[
 E_n\ge\frac{1+\sqrt{29}+\sqrt6}{6}.
\]
Since $\sqrt{174}>211/16$,
\[
 (\sqrt{29}+\sqrt6)^2
 =35+2\sqrt{174}>\frac{491}{8}>
 \left(\frac{47}{6}\right)^2.
\]
Thus $E_n>53/36$, and
$48E_n^2>2809/27>104$.

For the lower bound, first observe that
\[
 \gamma_n<\frac{41}{16}
\]
for all $n\ge2$.  This is immediate for $n=2$.  For $n\ge3$, the inequality is equivalent to
\[
 8\sqrt{17n^2-26n+9}<33n-24;
\]
after squaring, the difference between the right and left sides is
$n^2+80n>0$.  Hence $L_n<249/16$.  It follows that
\[
 D_n<\frac{257}{64},
 \qquad
 E_n<\frac{289}{192}.
\]
Indeed,
\[
 D_n^2<\frac{297}{32}+\frac{\sqrt{747}}4
 <\frac{129}{8}<\left(\frac{257}{64}\right)^2.
\]

Put
\[
 \widehat E_n=E_n+\frac1{160},
 \qquad
 t_n=\frac{\sigma_n}{\widehat E_n}.
\]
The preceding estimates give
\[
 \frac{47}{32}<\widehat E_n<\frac{1451}{960}.
\]
Moreover, $\widehat E_n>5/4$ and
\[
 t_n<\frac45\sigma_n\le\frac n{10};
\]
the last inequality is equivalent to
$32\le n^3(n+2)$.  Therefore, for $0\le t\le t_n$,
\[
 \frac{t^2}{16n}\le\frac t{160},
 \qquad
 P_n(t)\ge\sigma_n-\widehat E_nt.
\]
In particular, $P_n(t_n)\ge0$.  Since $P_n$ is strictly decreasing on $[0,\infty)$ and $r_n$ is its unique positive zero, it follows that $t_n\le r_n$.
All arguments of the sine and cosine below are smaller than $1/4$.  The elementary inequalities
$\cos y\ge1-y^2/2$ and
$\sin y/y\ge1-y^2/6$, followed by Bernoulli's inequality, give
\[
 \cos\!\left(\sqrt{\frac{3}{2n}}t\right)
 \left[
 \frac{\sin\!\left(\sqrt{\frac{n+1}{2n}}t\right)}
 {\sqrt{\frac{n+1}{2n}}t}
 \right]^{n-1}
 \ge1-c_nt^2,
\]
where
\[
 c_n=\frac{n^2+8}{12n}.
\]
Consequently,
\begin{align*}
 \varepsilon_n
 &\ge n\int_0^{t_n}
 t(\sigma_n-\widehat E_nt)^2(1-c_nt^2)\dd t\\
 &=\frac{1}{48\widehat E_n^2n(n+2)^2}
 \left(1-\frac{c_n\sigma_n^2}{5\widehat E_n^2}\right).
\end{align*}
A direct calculation gives
\[
 c_n\sigma_n^2
 =\frac{n^2+8}{24n^2(n+2)}\le\frac1{32}.
\]
Using the bounds for $\widehat E_n$,
\[
 1-\frac{c_n\sigma_n^2}{5\widehat E_n^2}
 >1-\frac{32}{11045}=
 \frac{11013}{11045}.
\]
Therefore the coefficient multiplying
$[n(n+2)^2]^{-1}$ is larger than
\[
 \frac{11013}{11045}\frac{1}{48(1451/960)^2}
 =\frac{42289920}{4650830809}>\frac1{110}.
\]
This proves the lower bound.  Finally,
$n+2\le2n$ gives the cubic estimate.
\end{proof}

\begin{proposition}\label{prop:epsilon-asymptotic}
One has
\[
 \lim_{n\to\infty}n^3\varepsilon_n
 =C_\infty
 =\frac{1}{48E_\infty^2}
 \approx0.00920130564911321,
\]
where
\[
 \gamma_\infty=\frac{1+\sqrt{17}}2,
 \qquad
 L_\infty=13+\gamma_\infty,
\]
\[
 D_\infty=\frac{\sqrt{L_\infty}+\sqrt3}{\sqrt2},
 \qquad
 E_\infty=\frac{1+2D_\infty}{6}.
\]
\end{proposition}

\begin{proof}
Put $t=y/n$ in \eqref{eq:epsilon-definition}.  Then
\[
 n\sigma_n\longrightarrow\frac1{\sqrt2},
 \qquad
 nr_n\longrightarrow\frac1{\sqrt2E_\infty},
\]
and
\[
 nP_n(y/n)\longrightarrow
 \frac1{\sqrt2}-E_\infty y.
\]
The sine and cosine factors converge uniformly to $1$ on the rescaled integration interval.  Dominated convergence gives
\begin{align*}
 \lim_{n\to\infty}n^3\varepsilon_n
 &=\int_0^{1/(\sqrt2E_\infty)}
 y\left(\frac1{\sqrt2}-E_\infty y\right)^2\dd y\\
 &=\frac1{48E_\infty^2}.
\end{align*}
\end{proof}

The first numerical values of the two principal coefficients are
\[
\begin{array}{c|c|c}
 n&\varepsilon_n&\delta_n\\ \hline
 2&2.9813775716\times10^{-4}&2.0752596325\times10^{-5}\\
 3&1.2243274270\times10^{-4}&4.6413452902\times10^{-6}\\
 4&6.3822649593\times10^{-5}&1.5173252326\times10^{-6}\\
 5&3.7527417681\times10^{-5}&6.1093306075\times10^{-7}\\
 6&2.3948629062\times10^{-5}&2.8349972676\times10^{-7}\\
 10&6.3884106880\times10^{-6}&2.9913619295\times10^{-8}.
\end{array}
\]

\end{document}